\documentclass[preprint,11pt]{elsarticle}
\usepackage{amsthm, amsmath, amssymb}
\usepackage{thmtools}
\usepackage{tikz}

\usepackage{graphicx}

\usepackage{caption}
\usepackage{subcaption}
\usepackage{makecell}
\usepackage{nicematrix}

\usepackage{enumitem}
\usepackage[left=1.5cm, right=1.5cm, top=2cm, bottom=2cm]{geometry}

\newtheorem{thm}{Theorem}[section]
\newtheorem{lem}[thm]{Lemma}

\newtheorem{cor}[thm]{Corollary}

\theoremstyle{definition}
\newtheorem{rem}[thm]{Remark}
\newtheorem{defin}[thm]{Definition}
\newtheorem{ex}[thm]{Example}

\begin{document}
\begin{frontmatter}

\title{Perfect stationary solutions of reaction-diffusion equations on lattices and regular graphs}

\author[inst1]{Vladim\'{i}r \v{S}v\'{i}gler\corref{cor1}}
\ead{sviglerv@kma.zcu.cz}
\cortext[cor1]{Corresponding author}

\affiliation[inst1]{organization={Department of Mathematics \& NTIS, Faculty of Applied Sciences, University of West Bohemia},
            addressline={Univerzitni~8}, 
            city={Pilsen},
            postcode={30100}, 
            state={Czech Republic}
            }

\author[inst1]{Jon\'{a}\v{s} Volek}
\ead{volek1@kma.zcu.cz}

\begin{abstract}
Reaction-diffusion equations on infinite graphs can have an infinite number of stationary solutions. These solutions are generally described as roots of a countable system of algebraic equations. As a generalization of periodic stationary solutions we present perfect stationary solutions, a special class of solutions with finite range in which the neighborhood values are determined precisely by the value of the central vertex. The focus on the solutions which attain a finite number of values enables us to rewrite the countable algebraic system to a finite one.  In this work, we define the notion of perfect stationary solutions and show its elementary properties. We further present results from the theory of perfect colorings in order to prove the existence of the solutions in the square, triangular and hexagonal grids; as a byproduct, the existence of uncountable number of two-valued stationary solutions on these grids is shown. These two-valued solutions can form highly aperiodic and highly irregular patterns. Finally, an application to a bistable reaction-diffusion equation on a square grid is presented. 
\end{abstract}

\begin{keyword}
regular grids \sep lattice differential equations \sep graph differential equations \sep stationary solutions \sep perfect colorings
\MSC 05C15 \sep 34A33 \sep 34B45 \sep 39A12 \sep 34B45
\end{keyword}

\end{frontmatter}

\section{Introduction}
We study a specific class of stationary solutions of graph and lattice reaction-diffusion differential equations with finite range. Let $G = (V,E)$ be generally an undirected graph with possibly infinite countable set of vertices $|V| \leq \infty$ and a set of edges $E \subset \genfrac(){0pt}{2}{V}{2}$, we consider the following (possibly countable) system of graph (reaction-diffusion) differential equations (GDEs):
\begin{equation}\label{e:GDE}
\dot{u}_i(t) = d \sum_{j \in N(i)}(u_j(t)-u_i(t))+f(u_i(t)), \quad i \in V, \quad t>0,
\end{equation}
in which $ d>0 $ is the diffusion parameter, $ N(i) := \{ j \in V \, |  \, \{ i,j\} \in E \} $ is the neighborhood of vertex $ i \in V $, and $ f: \mathbb{R} \to \mathbb{R} $ is a given continuous reaction function. 

This manuscript mainly treats a stationary problem for~\eqref{e:GDE}, i.e., the following (possibly countable) system of algebraic equations:
\begin{equation}\label{e:GDE-stationary}
	0 =	d \sum_{j \in N(i)}(u_j-u_i)+f(u_i), \quad i \in V.
\end{equation}
Given a solution $u: V \to \mathbb{R}$ of~\eqref{e:GDE-stationary}, we denote the set of its attained values by $\mathrm{Im}(u)$. We study solutions of~\eqref{e:GDE-stationary} with $ | \mathrm{Im}(u) | < \infty $ which possess additionally a specific regularity. Therefore, we necessarily restrict ourselves to $k$-regular graphs $G_k$, $ k \in \mathbb{N} $, in which $\text{deg}(i):= |N(i)|=k$ for every $i \in V$.

Let us note that the general setting of~\eqref{e:GDE} admits to consider graphs $ G $ being: finite graphs and study corresponding finite GDEs; regular infinite lattices and study corresponding lattice differential equations (LDEs), regular grids, or completely irregular infinite graphs. The equation~\eqref{e:GDE} on finite graphs is for example used to model patchy spatial environments~\cite{slavikLotkaVolterraCompetition2020, stehlikExponentialNumberStationary2017}. The equation~\eqref{e:GDE} on a lattice is a model used in crystallography, or describing dynamics of myelinated neurons and others, for overview see~\cite{hupkesTravelingWavesPattern2020}.

The stationary solutions of graph and lattice differential equations are in various ways significant to dynamical phenomena such as traveling waves of lattice differential equations. In a seminal paper~\cite{keenerPropagationItsFailure1987} Keener showed that there is an uncountable number of asymptotically stable stationary solutions of bistable lattice differential equation in a weak diffusion regime. These solutions are thought to prevent monotone traveling waves from moving. This phenomenon is nowadays known as \textit{pinning}, i.e., the existence of a nontrivial parameter region in which waves do not propagate, cf. continuous case for PDE's~\cite{fifeApproachSolutionsNonlinear1977}. The pinning phenomenon was examined on infinite paths~\cite{hoffmanUniversalityCrystallographicPinning2010}, multidimensional grids~\cite{hoffmanMultidimensionalStabilityWaves2015}, and infinite trees~\cite{hupkesPropagationReversalBistable2023, kouvarisTravelingPinnedFronts2012}.

Additional striking connection and motivation for this manuscript are multichromatic (nonmonotone) traveling waves of lattice differential equations. A profile of a wave often studied in bistable equations can be viewed as a solution connecting two stable homogeneous stationary solutions of the equation. It was revealed in~\cite{hupkesMultichromaticTravellingWaves2019} that bistable equations in fact admit traveling waves whose profile connects spatially periodic stationary solutions. The need for deeper understanding of the periodic solutions resulted in the discovery of their natural ordering~\cite{hupkesCountingOrderingPeriodic2019} and later to exact formulas for numbers of distinct solutions~\cite{sviglerBistableLatticeEquations2021}. 
The attempts to study periodic stationary solutions on multidimensional grids were not fruitful mainly due to an increment of parameters defining the periodicity (see Definition~\ref{d:square}), additionally the variance in the structure of the used grids (e.g., square, triangular, or hexagonal grids) leads to a limited portability of the results. 

The specific class of stationary solutions possessing some regularity examined in this paper are the so called \emph{perfect stationary solutions}. The term originates from their natural relationship with \emph{perfect colorings} of graphs (Definition~\ref{d:perfect-coloring}). The perfect stationary solutions preserve the crucial property of periodic solutions -- the solutions are roots of finite systems of algebraic equations. Moreover, every periodic stationary solution of a regular grid is a perfect stationary solution. However, perfect stationary solutions encompass also additional irregular solutions, see Figure~\ref{f:applications} for illustration. Generally speaking, focusing on the local properties of the stationary solutions instead of the global patterns leads to greater freedom. 

Specifically, a perfect stationary solution $ u: V \to \mathbb{R} $ of~\eqref{e:GDE} satisfies that the value (color) of the solution at a vertex uniquely determines the occurrence of other values (colors) in its neighborhood. It is crucial to emphasize that only a set of neighboring values (colors) is specified, not its precise structure. There are benefits of this approach which come at a cost. The benefits are: 
\begin{itemize}
\item Each periodic stationary solution on a regular grid is a perfect one. This follows from the fact that any periodic solution forms a regular tiling, see Theorem~\ref{t:periodic-is-perfect-sol}. 
\item The framework enables us to split some problems into simpler ones via merging of the colors, i.e., requiring certain solution values to be identical, see Example~\ref{ex:periodic-is-perfect} and Section~\ref{s:application}.
\item We can for example show that there is an uncountable number of two-value stationary solution for bistable dynamics on square, triangular or hexagonal grid, Theorems~\ref{t:square:grid},~\ref{t:triangular:grid}, and~\ref{t:hexagonal:grid}. One can compare this to a classical result of Keener \cite{keenerPropagationItsFailure1987} showing an uncountable number of stationary solutions on path whose range is not a finite set in general. 
\end{itemize}
The cost is that the proof of existence of perfect solutions requires to employ results from the theory of perfect colorings of graphs -- in addition to the classical problem of existence of a solution to a system of algebraic equations. Indeed, the existence of a perfect coloring answers the question whether a certain solution profile can be mapped to a given graph. 

We provide an extended overview of the structure of the paper in order to make orientation in this paper easier.  

\textbf{Section 2} serves to introduce preliminaries necessary for further exposition. We start with the colorings of regular graphs and follow with the notion of multisets. The section is concluded by the introduction of perfect colorings of regular graphs. 

In \textbf{Section 3}, we define a perfect stationary solution (Definition~\ref{d:perfect-solution}) together with a claim showing the connection between infinite- and finite dimensional system describing the perfect stationary solution (Theorem~\ref{t:main}). The subtle boundary between the perfect colorings (i.e., a graph-theoretical notion) and the perfect stationary solutions (i.e., a specific stationary solution of a dynamical system) is further examined. Namely, we show that every two-valued stationary solution of a GDE or a stationary solution of a GDE on a finite graph is necessarily perfect stationary solutions. We subsequently provide a counter-example showing the existence of a three-valued stationary solution which is not perfect. Technical lemmas which comment on a construction of new perfect stationary colorings and solutions from existing ones via merging of colors conclude this part of the manuscript.

Up to this point the only assumptions posed on the underlying graph were the regularity and the cardinality of the vertex set. In contrast, \textbf{Section 4} focuses on the existence and properties of perfect stationary colorings on lattices, namely: infinite path graph, square grid, triangular grid, and hexagonal grid. The situation is somewhat simple on an infinite path graph since every perfect stationary solution is a periodic one (Lemma~\ref{l:path}) and the existence of the coloring can be easily checked. However, considering two-dimensional grids opens completely new behavior. We define the periodic coloring on a grid and show that every periodic coloring can be handled as a perfect one (Theorem~\ref{t:periodic-is-perfect}) which is then readily developed into Theorem~\ref{t:periodic-is-perfect-sol} stating that each periodic stationary solution on any of these grids is also perfect. For each of the mentioned grids, we summarize results relating the spectral properties of the matrix $\mathbf{m}$ describing the neighborhood color sets and the actual existence of perfect colorings, i.e., the existence of a proper layout of the colors. Particular constructions of two-color perfect color further reveal uncountable number of possible layouts which cannot be transformed onto each other via graph automorphisms. 

The results of the previous sections are combined in \textbf{Section 5} in an example treating four-valued perfect stationary solutions of a Nagumo bistable GDE on a square grid. 

We provide concluding remarks and possible future directions in \textbf{Section 6}.

\section{Preliminaries}\label{s:prelim}

In this preliminary section we introduce needed definitions and statements on graph colorings and mainly, on a specific class of colorings of regular graphs -- the so-called perfect colorings.
	 
\paragraph{Colorings of regular graphs, mergers}

Let $ k \in \mathbb{N} $ and $ G_{k} = (V, E) $ be a $k$-regular graph. A coloring of $G_k$ is a map: $\Gamma: V \to C$ from the set of vertices $ V $ to a finite set of colors $ C := \{ 1, 2, \ldots, n \}  $. Given two colorings of a graph, they can possess a similar structure obtained by merging of colors. This is formalized in the next definition.
\begin{defin}\label{d:split-and-merger}
   Let $ k \in \mathbb{N} $, a $k$-regular graph $G_k$ be given and let $\Gamma_1$, $\Gamma_2$ be two colorings of $G_k$ with their respective color sets $C_1 \subsetneq C_2$ such that $\ell = |C_1| < |C_2| = n$.
   We say that $\Gamma_1$ is a \textit{merger} of $\Gamma_2$ if there exists a surjective map $\phi:C_2 \to C_1$ such that 
   \[
   \Gamma_1(i) = \phi(\Gamma_2(i)) \quad \text{for every} \quad  i \in V.
   \]
\end{defin}

\paragraph{Multisets}
Besides merely keeping track of adjacent colors, we also need information on their respective quantities. In order to handle this task, we introduce the following multiset formalism as in~\cite{syropoulosMathematicsMultisets2001}. The multiset is a tuple $M := (A,m)$ in which $A$ is an underlying set and $m: A \to \mathbb{N}_0$ is a mapping denoting the multiplicities of the elements of $A$. When there is no room for confusion we will freely represent the mapping $m$ as a vector $ m \in \mathbb{N}_0^{|A|}$ provided the ordering of the set $A$ is obvious. In case we want to enumerate a multiset we use square brackets. For example, $M = (\{ a, b, c\}, (a \mapsto 2, b \mapsto 0,c \mapsto 1)) = [ a,a,c]$. If we consider the ordering $a<b<c$, then one has $ m = (2,0,1) \in \mathbb{N}_0^{3} $.

\begin{rem}
A set $ A $ is a special case of multiset with $m := \chi_A$ being an indicator function. 
\end{rem}

\paragraph{Perfect colorings}

The focal class of colorings are perfect colorings in which the colors in the neighborhood of each vertex are precisely specified by its color. This can be represented by a matrix in the following manner. Let $ G_{k} $ be a $k$-regular graph, $ C = \{ 1,2, \ldots , n \} $ be a set of colors, and
\[
m_{i} = ( (m_{i})_{1}, (m_{i})_{2}, \ldots, (m_{i})_{n} ) \in \mathbb{N}_{0}^{n}
\]
be the vector form of multiset representing the color structure of the neighborhood of vertices with color $ i \in C $, i.e., $ m_{i} $ contains as $ (m_{i})_{j} $ the number of neighbors with color $ j \in C $ of the vertex with color $ i \in C $, then the structure of the perfect coloring is represented by the coloring matrix
\[
\mathbf{m} = \left[
\begin{array}{ccc}
&m_{1}& \\
\hline
&m_{2}& \\
\hline
&\vdots& \\
\hline
&m_{n}&
\end{array}
\right].
\]
We simplify the notation to be consistent with the indexing of matrices as $ (m_{i})_{j}  = m_{ij}$, $ i,j=1,2,\ldots,n $. There has to necessarily be
\[
\sum_{j=1}^n m_{ij} = k \quad \text{for every} \quad i = 1, 2, \ldots, n,
\]
since each vertex has exactly $k$ neighbors in a $ k $-regular graph $ G_{k} $.

The following definition then formalizes the notion of perfect coloring.
	
	\begin{defin}\label{d:perfect-coloring}
		Let $n,k \in \mathbb{N}$, $C = \{1,2, \ldots, n\}$, $G_k$ be a $k$-regular graph, and $ \mathbf{m} \in \mathbb{N}_{0}^{n \times n} $ be such that $ \sum\nolimits_{j=1}^{n} m_{ij} = k $ holds for all $ i = 1,2,\ldots, n $. 
        Then the coloring $\Gamma$ of $G_k$ is called \emph{$\mathbf{m}$-perfect} if
		\[
			[\Gamma(j)]_{j \in N(i)} = (C,m_{\Gamma(i)\cdot}), \quad \text{for every} \quad i \in V,
		\]
		in which $m_{\Gamma(i)\cdot}$ is $\Gamma(i)$-th row of the matrix $\mathbf{m} \in \mathbb{N}_{0}^{n \times n}$. 
	\end{defin}

	\begin{ex}\label{ex:path}
		Let $G_2$ be an infinite path with vertices $V = \mathbb{Z}$. Then the coloring 
		\begin{equation}\label{ex:path:gamma}
			\Gamma(i)=\begin{cases}
				1, & i \; \mathrm{mod} \; 3 = 0, \\
				2, & i \; \mathrm{otherwise},
			\end{cases}
        \end{equation}
		depicted in Figure~\ref{f:ex:path} is $\mathbf{m}$-perfect with 
		\begin{equation}\label{ex:path:m}
			\mathbf{m} = \left[
			\begin{array}{cc} 
				0 & 2 \\
				1 & 1
			\end{array}
			\right],
		\end{equation}
		in which the colors $ C = \{1,2\} $ are white and black, respectively. Each white vertex has exactly two black neighbors and each black vertex has one black and one white neighbor. Note that the entries in both rows $m_{11}+m_{12} = m_{21} + m_{22}=2$ add up to $2$, since the path graph is $2$-regular.
				
		\begin{figure}			
			\centering
            \includegraphics[width=.6\textwidth]{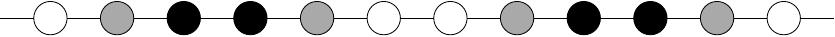}
			\caption{Example of a perfect coloring of an infinite path in Example~\ref{ex:path}.}
			\label{f:ex:path}
		\end{figure}
	\end{ex}

The following lemma describes the necessary sign structure of a matrix $\mathbf{m}$ describing an $\mathbf{m}$-perfect coloring of a graph. 

\begin{lem}\label{l:symmetry}
    Let $ n,k \in \mathbb{N} $, $ G_{k} $ be a $k$-regular graph, and $\mathbf{m} \in \mathbb{N}_{0}^{n \times n}$ be such that $\sum_{j=1}^n m_{ij}=k$ holds for all $i = 1,2, \ldots, n$. If there exists an $\mathbf{m}$-perfect coloring of $ G_{k} $, then the sign matrix of $\mathbf{m}$ is symmetric. 
\end{lem}
\begin{proof}
    The matrix $\mathbf{m}$ has either zero or positive entries. Consider $i,j=1, \ldots, n$, $i \neq j$. If $m_{ij} > 0$, the existence of $\mathbf{m}$-perfect coloring yields that a vertex of color $i$ has $m_{ij} > 0$ neighbors of color $j$. Since $ G_{k} $ is undirected, a vertex of color $j$ must have at least one neighbor of color $i$, which concludes $m_{ji} > 0 $, i.e., $ \mathrm{sign}(m_{ij}) = \mathrm{sign}(m_{ji}) $.
\end{proof}

We conclude this section with a technical lemma which states that given a perfect coloring, there may be a simpler one obtained via merging of colors. Further, it gives a formula on how to construct the matrix of the simpler perfect coloring. See Lemma~\ref{l:ss-merger} for specific use.
\begin{lem}\label{l:perfect-merger}
    Let $ k \in \mathbb{N} $, a $k$-regular graph $G_k$ be given and let $\Gamma_1$, $\Gamma_2$ be two colorings of $G_k$ with their respective color sets $C_1 \subsetneq C_2$ such that $\ell = |C_1| < |C_2| = n$. Let $\Gamma_2$ be perfect with the matrix $\mathbf{m}_2 \in \mathbb{N}_0^{n \times n}$. If $\Gamma_1$ is perfect and it is a merger of $\Gamma_2$ with a merging surjective map $\phi:C_2 \to C_1$ satisfying $\varphi(i) = i$, $i= 1, \ldots, \ell$, then $\Gamma_1$ is $\mathbf{m}_1$-perfect with $\mathbf{m}_1 \in \mathbb{N}_0^{\ell \times \ell}$ given by
    \begin{equation}\label{e:perfect-merger}
        (\mathbf{m}_1)_{ij} = \sum_{s \in \phi_{-1}(j)} (\mathbf{m}_2)_{is},  \quad i,j \in C_1.
    \end{equation}
\end{lem}

\begin{proof}
The proof of~\eqref{e:perfect-merger} is immediate, since the entries in the new matrix $\mathbf{m}_1$ are just total numbers of the colors merged together.
\end{proof}

\section{Perfect stationary solutions and their properties}\label{s:perfect}

Here, we firstly introduce the notion of perfect stationary solutions of~\eqref{e:GDE} via the perfect colorings of an underlying graph presented in the previous Section~\ref{s:prelim}. We then state a theorem on the existence of perfect stationary solutions and discuss their relation to corresponding colorings.

\begin{defin}\label{d:perfect-solution}
Let $n,k \in \mathbb{N}$, $G_k$ be a $k$-regular graph, and let $\mathbf{m} := (m_{ij}) \in \mathbb{N}_{0}^{n \times n}$ be a matrix such that $\sum_{j=1}^n m_{ij}=k$ holds for all $i = 1, \ldots, n$. A stationary solution $u$ of~\eqref{e:GDE} is called an \emph{$\mathbf{m}$-perfect stationary solution} provided there exists an $\mathbf{m}$-perfect coloring $\Gamma$ of the graph $G_k$ and a surjection $ v: C \to \mathrm{Im}(u) $ such that
\[
v_{\Gamma(i)} = u_{i} \quad \text{for every} \quad i \in V .
\]
\end{defin}

The following theorem describes a way how to construct the mapping $ v: C \to \mathrm{Im}(u) $ from Definition~\ref{d:perfect-solution} by a solution of a specific finite system of algebraic equations which employs the coloring matrix $ \mathbf{m} $.

\begin{thm}\label{t:main}
    Let $n,k \in \mathbb{N}$, $G_k$ be a $k$-regular graph, and let $\mathbf{m} := (m_{ij}) \in \mathbb{N}_{0}^{n \times n}$ be such that $\sum_{j=1}^n m_{ij}=k$ holds for all $i = 1,2, \ldots, n$. Let $v \in \mathbb{R}^n$ be a solution of the equation 
    \begin{equation}\label{e:finite-stationary}
        0 = d(\mathbf{m} \, \cdot \, v - k v) + F(v).
    \end{equation}
    If there exists an $\mathbf{m}$-perfect coloring $\Gamma$ of the graph $G_k$, then there exists an $ \mathbf{m} $-perfect stationary solution $u: V \to \mathbb{R} $ of~\eqref{e:GDE} given by
    \begin{equation}\label{e:finite-surjection}
    v_{\Gamma(i)} = u_{i}, \quad i \in V .
    \end{equation}
    Moreover, the vector $v$ is an asymptotically stable stationary solution of the system 
    \begin{equation}\label{e:finite-dynamic}
        \dot{v}(t) = d(\mathbf{m} \, \cdot \, v(t) - k v(t)) + F(v(t))
    \end{equation}
    if and only if $ u $ is an asymptotically stable stationary solution of~\eqref{e:GDE} in the $\ell^\infty$-norm. 
\end{thm}

\begin{proof}
The formula~\eqref{e:finite-surjection} directly defines a map $u:V \to \mathbb{R}$ and if $u$ is a solution of~\eqref{e:GDE-stationary} then it is an $\mathbf{m}$-perfect stationary solution (cf.~Definition~\ref{d:perfect-solution}). To conclude the first part of the proof, we claim that $u$ is actually a solution of~\eqref{e:GDE-stationary}. Indeed, let $ i \in V $ be arbitrary, then the corresponding $ u_{i} $ satisfies~\eqref{e:GDE-stationary} by the $ \Gamma(i) $-th equation of finite system~\eqref{e:finite-stationary}.

The stability claim follows from~\cite[Lemma 1]{hupkesCountingOrderingPeriodic2019}. 
\end{proof}

\begin{rem} \label{r:finite-image-vs-perfect}
    By definition, for each $\mathbf{m}$-perfect stationary solution $u$ of~\eqref{e:GDE} the inequality $|\mathrm{Im}(u)| \le n$ holds. Note that the case $|\mathrm{Im}(u)| < n $ can commonly occur even for $ n = k = 2 $ (i.e., different colors from the perfect coloring can be both mapped to the same value of the stationary solution $ u $, see the following Example~\ref{ex:Im(u)<n}).
\end{rem}

\begin{ex}\label{ex:Im(u)<n}
Let $ G_{2} $ be the infinite path ($ V = \mathbb{Z} $) colored in the alternating manner by
\[
\Gamma(i) = \left\lbrace
\begin{array}{ll}
1, & i = 2l+1 \ \text{for} \ l \in \mathbb{Z}, \\
2, & i = 2l \ \text{for} \ l \in \mathbb{Z}.
\end{array}
\right.
\]
Then
\[
\mathbf{m} = \left[
\begin{array}{cc}
0 & 2 \\
2 & 0
\end{array}
\right]
\]
and~\eqref{e:finite-stationary} is then equivalent to
\begin{equation}\label{e:Im(u)<n}
\left\lbrace
\begin{array}{rcl}
0 & = & d(2v_{2}-2v_{1}) + f(v_{1}), \\
0 & = & d(2v_{1}-2v_{2}) + f(v_{2}).
\end{array}
\right.    
\end{equation}
Consider the prototypical bistable nonlinearity $ f(u) = u(1-u)\left(u-\tfrac{1}{2}\right) $, i.e., a symmetric case. From~\cite{stehlikExponentialNumberStationary2017} we deduce that for large $ d \gg 1 $ there are three homogeneous solutions $ v = (0,0) $, $ v = \left(\tfrac{1}{2},\tfrac{1}{2}\right) $, and $ v = (1,1) $ of~\eqref{e:Im(u)<n} (corresponding to spatially homogeneous steady states $ u_{i} \equiv 0 $, $ u_{i} \equiv \tfrac{1}{2} $, and $ u_{i} \equiv 1 $ of~\eqref{e:GDE-stationary}, respectively) from which another (heterogeneous) solutions bifurcate for decreasing $ d \searrow 0 $ and finally there are in total $ 3^{2}=9 $ solutions of~\eqref{e:Im(u)<n} for sufficiently small $ 0 < d \ll 1 $.

In particular, focusing on the solution
\[
v = \left( \tfrac{1}{2}, \tfrac{1}{2} \right),
\]
which corresponds to homogeneous $\mathbf{m}$-perfect stationary solution
\[
u = \left( \ldots,\tfrac{1}{2}, \tfrac{1}{2}, \tfrac{1}{2}, \tfrac{1}{2}, \tfrac{1}{2}, \tfrac{1}{2}, \ldots \right)
\]
of~\eqref{e:GDE}, we point out that $ |\mathrm{Im}(u)|=1<n=2 $ in this case. This is same for the two other homogeneous solutions. Furthermore, by the symmetry of $ f(u) $ and since $ f'\left( \tfrac{1}{2} \right) = \tfrac{1}{4} $, we deduce that for $ d < \tfrac{1}{16} $ (i.e., $ f'\left( \tfrac{1}{2} \right) > 4d $) there exists a solution $ \tilde{v} = (\tilde{v}_{1},\tilde{v}_{2}) $ of~\eqref{e:Im(u)<n} with $ \tilde{v}_{1} \in \left( 0, \tfrac{1}{2} \right) $ and $ \tilde{v}_{2} \in \left( \tfrac{1}{2}, 1 \right) $ which bifurcates from $ v = \left( \tfrac{1}{2}, \tfrac{1}{2} \right) $ at $ d=\tfrac{1}{16} $. The corresponding heterogeneous $\mathbf{m}$-perfect stationary solution of~\eqref{e:GDE} is
\[
\tilde{u} = \left( \ldots, \tilde{v}_{1}, \tilde{v}_{2}, \tilde{v}_{1}, \tilde{v}_{2}, \tilde{v}_{1}, \tilde{v}_{2}, \ldots \right)
\]
and there is $ |\mathrm{Im}(\tilde{u})|=2=n $ in this case (similarly for other heterogeneous solutions).
\end{ex}

\begin{rem}

With a given matrix $ \mathbf{m} \in \mathbb{N}_{0}^{n \times n} $, $ n \in \mathbb{N} $, the problem of existence of $ \mathbf{m} $-perfect stationary solutions of~\eqref{e:GDE} has by Theorem~\ref{t:main} the following three ingredients:
\begin{enumerate}[label=\emph{\Roman*.}]
    \item \emph{The existence of a solution of finite algebraic system~\eqref{e:finite-stationary}.} To analyze the solvability, the number and structure of solutions of~\eqref{e:finite-stationary}, one can apply any method as algebraic, topological, variational, or monotonicity methods, etc. (see, e.g.,~\cite{drabekMethodsNonlinearAnalysis2013}). We sum up several introductory comments on the relation of perfect solutions and solutions with finite image $ |\mathrm{Im}(u)|< \infty $ in the rest of this section.
    \item \emph{The existence of an $ \mathbf{m} $-perfect coloring of $ G_{k} $.} In this manuscript we are however mainly interested in the application of graph-theoretical results on perfect colorings. In Section~\ref{s:infinite-graphs} we discuss existence of $ \mathbf{m} $-perfect colorings of regular squa\-re/tri\-an\-gu\-lar/he\-xa\-go\-nal grids $ G_{k} $ and their relation to periodic solutions. We also discuss the number of such colorings and highlight that there can exist even uncountably many of such colorings with only two colors. This possibly yields a rich structure of existing stationary patterns of~\eqref{e:GDE}, see Section~\ref{s:application}. 
    \item \emph{The correspondence of values and colors.} Example~\ref{ex:Im(u)<n} shows that there can be less values in $ \mathrm{Im}(v) $ of the solution $ v $ of finite system~\eqref{e:finite-stationary} than colors of the perfect coloring. We resolve the issue by mergers of the colorings from Lemma~\ref{l:perfect-merger} -- i.e., creating simpler perfect colorings from the existing ones -- and by forthcoming Lemma~\ref{l:ss-merger}. See also Remark~\ref{r:periodic-is-perfect-sol} and Section~\ref{s:application} for specific application and comments.
\end{enumerate}
\end{rem}

In the next part of this section we discuss whether the opposite implication in Theorem~\ref{t:main} works as well:\\[1ex]
\indent \emph{Given a stationary solution $ u : V \to \mathbb{R} $ of~\eqref{e:GDE} on a $k$-regular graph $ G_{k} $ such that $ | \mathrm{Im}(u) | < \infty $, does it define an $ \mathbf{m} $-perfect coloring $ \Gamma $ of $ G_{k} $ with a finite $ \mathbf{m} \in \mathbb{N}_{0}^{n \times n} $, $ n \in \mathbb{N} $, such that $ u_{i} \neq u_{j} $ implies $ \Gamma(i) \neq \Gamma(j) $?}
\\[1ex]
Firstly, we show in the next lemma that it is true in the case of solutions which posses only two values.

\begin{lem}\label{l:im-equals-n}
Let $ k \in \mathbb{N} $ and $ G_{k} $ be given and $u$ be a stationary solution of~\eqref{e:GDE} such that $ |\mathrm{Im}(u)|=2 $. Then there exists $\mathbf{m} \in \mathbb{N}_{0}^{2 \times 2}$ such that $\sum_{j=1}^n m_{ij}=k$ for $ i = 1,2 $, and an $\mathbf{m}$-perfect coloring $\Gamma$ of the graph $G_k$ such that $ v $ given by $ v_{i} = u_{\Gamma^{-1}(i)} $, $ i \in V $, solves~\eqref{e:finite-stationary}.
\end{lem}

\begin{proof}
Let $ \mathrm{Im}(u) = \{ \alpha, \beta \} $, $ \alpha \neq \beta $. Denote
\[
m_{i\alpha} = | \{ j \in N(i): \ u_{j} = \alpha \} |, \quad m_{i\beta} = | \{ j \in N(i): \ u_{j} = \beta \} |
\]
We claim that there exists $ m_{\alpha\alpha} \in \mathbb{N}_{0} $ and $ m_{\alpha\beta} \in \mathbb{N}_{0} $ such that
\[
m_{i\alpha} = m_{\alpha\alpha}, \quad m_{i\beta} = m_{\alpha\beta} \quad \text{for every} \quad i \in V \quad \text{such that} \quad u_{i} = \alpha
\]
(and analogically for vertices with $ u_{i} = \beta $). Indeed, assume by contradiction that, without loss of generality, there are $ i, \tilde{i} \in V $ such that $ u_{i} = u_{\tilde{i}} = \alpha $ and $ m_{i\alpha} \neq m_{\tilde{i}\alpha} $ (which also implies that $ m_{i\beta} \neq m_{\tilde{i}\beta} $ since $ G_{k} $ is $k$-regular). Then GDE~\eqref{e:GDE} yields that
\[
0 = d(m_{i\alpha}\alpha + m_{i\beta} \beta - k\alpha) + f(\alpha)
\quad \text{and} \quad
0 = d(m_{\tilde{i}\alpha}\alpha + m_{\tilde{i}\beta} \beta - k\alpha) + f(\alpha)
\]
which together with $k$-regularity of $ G_{k} $ verify the following computations:
\[
\begin{array}{rcl}
m_{i\alpha}\alpha + m_{i\beta} \beta &=& m_{\tilde{i}\alpha}\alpha + m_{\tilde{i}\beta} \beta \\
m_{i\alpha}\alpha + (k-m_{i\alpha}) \beta &=& m_{\tilde{i}\alpha}\alpha + (k-m_{\tilde{i}\alpha}) \beta \\
(m_{i\alpha} - m_{\tilde{i}\alpha})(\alpha - \beta) &=& 0.
\end{array}
\]
Since $ \alpha \neq \beta $, there must be $ m_{i\alpha} = m_{\tilde{i}\alpha} $, a contradiction. Therefore, $ m_{\alpha\alpha} $ and $ m_{\alpha\beta} = k - m_{\alpha\alpha}$ are correctly defined. Identical reasoning can be done for vertices with value $\beta$, defining $m_{\beta\beta}$ and $m_{\beta\alpha} = k - m_{\beta\beta}$.

Setting the coloring $ \Gamma $ of $ G_{k} $ as
\[
\Gamma(i) = \left\lbrace
\begin{array}{ll}
1, & \text{if } u_{i} = \alpha, \\
2, & \text{if } u_{i} = \beta,
\end{array}
\right.
\]
we can define
\[
\mathbf{m} = \left[
\begin{array}{cc}
m_{\alpha\alpha} & k- m_{\alpha\alpha} \\
k-m_{\beta\beta} & m_{\beta\beta}
\end{array}
\right].
\]
Consequently, the coloring $ \Gamma $ is $ \mathbf{m} $-perfect and one can easily check that $ v $ given by $ v_{i} = u_{\Gamma^{-1}(i)} $ solves~\eqref{e:finite-stationary}.
\end{proof}

Furthermore, every stationary solution of GDE~\eqref{e:GDE} on a finite graph $G_k$ (i.e., $ |V| < +\infty $) is obviously perfect.

\begin{lem}\label{l:finite-image}
Let $k \in \mathbb{N}$ and a $k$-regular graph $G_k$ be such that $|V| < +\infty$. Then for each stationary solution $u$ of~\eqref{e:GDE} on $G_k$ there exists a suitable matrix $\mathbf{m} \in \mathbb{N}_{0}^{n \times n}$, $ n \in \mathbb{N} $, such that $u$ is an $\mathbf{m}$-perfect stationary solution. 
\end{lem}

\begin{proof}
 It is sufficient to find a perfect coloring $\Gamma$ of the graph $G_k$ such that for each $i,j \in V$ the equality $\Gamma(i) = \Gamma(j)$ implies $u_i = u_j$. One can take the color set $ C := V $ and consider the trivial coloring $ \Gamma(i) = i $, $ \quad i \in V $, which is therefore an $ \mathbf{m} $-perfect coloring with $ \mathbf{m} $ being the adjacency matrix of $ G $. 
\end{proof}

\begin{rem}
\label{r:periodic-vs-perfect}
Although finite $k$-regular graphs form a rather small subset of all $k$-regular graphs, Lemma~\ref{l:finite-image} indicates an interesting relation between periodic and perfect solutions. Specifically, every periodic stationary solution on a regular grid, i.e., solution created by repeating a finite pattern, is perfect, see Theorem~\ref{t:periodic-is-perfect} -- note that the opposite implication is not generally valid, see~Section~\ref{s:application}. Nonetheless, since the notion of periodicity is graph-specific, we postpone the issue to the forthcoming Section~\ref{s:infinite-graphs} in which we discuss periodic and perfect colorings of square/triangular/hexagonal grids.
\end{rem}

Unfortunately, a stationary solution of~\eqref{e:GDE} with $ | \mathrm{Im}(u) | < \infty $ on infinite graph $ G_{k} $ which possesses at least three values need not be $ \mathbf{m} $-perfect with any finite matrix $\mathbf{m} $ as the next example shows. 

\begin{ex}\label{ex:tree}
Let us have a biinfinite binary tree which is a $3$-regular graph $G_3 := T_2$. Let $u$ be a stationary solution of~\eqref{e:GDE} such that $\mathrm{Im}(u) = \{ a,b,c \}$, $a,b,c \in \mathbb{R}$. Finally, assume that there exists an infinite path $P$ in the graph $G_3$ indexed by integers $\mathbb{Z}$ such that  
\begin{equation}\label{e:ex:tree:coloring}
u_i = \begin{cases}
    a, & i = \pm \frac{1}{2}n(n+9) \quad \text{for} \quad n \in \mathbb{N}_0, \\
    b, & i = \pm \frac{1}{2}n(n+9)\pm 1 \quad \text{for} \quad n \in \mathbb{N}_0, \\
    c, & \text{otherwise},
\end{cases}
\end{equation}
see Figure~\ref{f:ex:tree}. For the sake of clarity, we color the vertices with the value $a$ as $A$, the vertices with the values $b$ as $B$, the outer vertices with the value $c$ as $C_1$ and finally, the inner vertices with the value $c$ as $C_2$, see Figure~\ref{f:ex:tree}. Let us denote this coloring $\widetilde{\Gamma}$. 
    \begin{figure}[ht]			
			\centering
           \includegraphics[width=\textwidth]{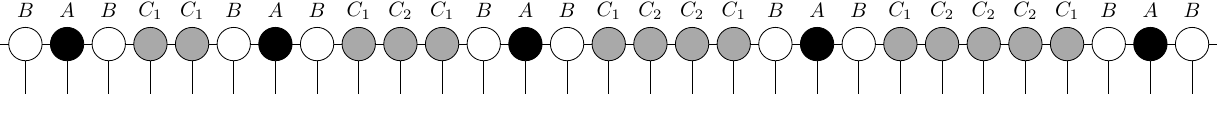} 
			\caption{An illustration of a finite-image stationary solution $u$ ($\mathrm{Im}(u)< +\infty$) of~\eqref{e:GDE} prescribed by~\eqref{e:ex:tree:coloring} on an infinite binary tree. The pattern develops accordingly in the negative direction. This solutions is not perfect, see Example~\ref{ex:tree}.}
			\label{f:ex:tree}
		\end{figure}

Furthermore, let us assume the following distribution of the values $a,b,c$ throughout the whole tree
\begin{equation}\label{e:ex:tree:coloring2}
[u_j]_{j \in N(i)} = \begin{cases}
[a,b,b], & \widetilde{\Gamma}(i)=A, \\
[a,c,c], & \widetilde{\Gamma}(i)=B, \\
[b,b,c], & \widetilde{\Gamma}(i)=C_1, \\
[a,c,c], & \widetilde{\Gamma}(i)=C_2,
\end{cases}
\end{equation}
by which we iteratively extend the coloring $\widetilde{\Gamma}$ of the path $P$ to a coloring $\Gamma$ of the whole graph $G_3$. 
 This distribution of the values exists for given values along the path $P$ and it is unique up to isomorphism. The growing islands of the vertices with the value $c$ along $P$ contradict the existence of $\mathbf{m}$-perfect coloring corresponding to the stationary solution $u$. Let us look into this in greater detail. Start at the leftmost vertex in Figure~\ref{f:ex:tree}. The coloring of the two-vertex and the three-vertex gray islands can be perfect with the prescribed neighborhood sets. The next gray island consisting of the sequence $-\!C_1\!-\!C_2\!-\!C_2\!-\!C_1\!-$ however poses a first problem. In order to preserve a perfect coloring we need to introduce a new color, say $C_3$, instead of $C_2$. The introduction of a new color is necessary in each of the subsequent islands and thus, the number of colors grows infinitely. 

 Finally, we show that $u$ can indeed exist. The values $a,b,c$ must satisfy
 \begin{align}
     a+2b-3a+f(a) &= 0, \nonumber  \\
     a+2c-3b+f(b) &= 0, \nonumber \\
     2b+c-3c+f(c) &= 0, \label{e:ex:tree:1} \\
     a+2c-3c+f(c) &= 0. \label{e:ex:tree:2}
 \end{align}
Given any triplet $a,b,c$ satisfying $-a+2b-c=0$, i.e., the difference of~\eqref{e:ex:tree:1} and~\eqref{e:ex:tree:2}, there exists a nonlinearity $f$ such that $u$ is a stationary solution of~\eqref{e:GDE} on $G_3$ with $|\mathrm{Im}(u)|< +\infty$ and which is not perfect with finite number of colors.

\end{ex}

We conclude the treatise of perfect stationary solutions with a short note about the connection of the merging of perfect colorings a their respective perfect stationary solutions and the focal claim that every periodic stationary solution is a perfect stationary solution. 
\begin{lem}\label{l:ss-merger}
Let $ k \in \mathbb{N} $, a $k$-regular graph $G_k$ be given and let $\Gamma_1$, $\Gamma_2$ be two colorings of $G_k$ with their respective color sets $C_1 \subsetneq C_2$ such that $\ell = |C_1| < |C_2| = n$.
Let $ \Gamma_{1} $ be $\mathbf{m}_1$-perfect, $ \Gamma_{2} $ be $\mathbf{m}_2$-perfect, and let $\Gamma_1$ be a merger of $\Gamma_2$ with a merging surjective map $\phi: C_{2} \to C_{1} $. If $v_1 \in \mathbb{R}^{\ell} $ is a solution of 
\begin{equation} \label{e:merger-m1}
0 = d(\mathbf{m}_1 \, \cdot \, v_1 - k v_1) + F(v_1), 
\end{equation}
then $v_2 \in \mathbb{R}^{n}$ given by
\[
(v_{2})_i = (v_{1})_{\phi(i)}, \quad i = 1,2, \ldots, n,
\]
is a solution of 
\begin{equation} \label{e:merger-m2}
0 = d(\mathbf{m}_2 \, \cdot \, v_2 - k v_2) + F(v_2),
\end{equation}
in which the matrices $\mathbf{m}_1$ and $\mathbf{m}_2$ are related via~\eqref{e:perfect-merger}.
\end{lem}

\begin{proof}
   Without loss of generality assume that the vectors $v_1$ and $v_2$ are indexed by the color sets $C_1$ and $C_2$ respectively. It is now sufficient to realize that~\eqref{e:merger-m1} originates from the system~\eqref{e:merger-m2} by requiring certain entries of $v_2$ to be equal to each other. Indeed, the formula~\eqref{e:perfect-merger} effectively requires the entries $i$ of $v_2$ which are in the common preimages of $\phi_{-1}(j)$, $j \in C_1$, to be equal.
\end{proof}

We present a specific application of Lemma~\ref{l:ss-merger} in Section~\ref{s:application}.

\section{Perfect colorings of infinite graphs}\label{s:infinite-graphs}
The forthcoming section is divided into two parts. We first handle the question of the existence of perfect colorings on an infinite path graph. We then turn our attention to particular two-dimensional infinite grids: square, triangular, and hexagonal grid. Namely, we examine the connection between periodic and perfect colorings, Theorems~\ref{t:periodic-is-perfect} and~\ref{t:periodic-is-perfect-sol}, and the existence and the cardinality of perfect colorings on the grids.

\subsection{Path graph}
We start the discussion about the existence of perfect stationary solution on particular infinite graphs with the simplest case -- an infinite path graph. The existence of periodic stationary solutions on the path graph has been examined in~\cite{hupkesCountingOrderingPeriodic2019}. We provide a sufficient condition on the existence of a perfect coloring on a path graph and connect the notion of perfect solutions to the periodic ones. 

The vertex set of the path graph is usually identified with the set of integers $V = \mathbb{Z}$ and the set of edges $E = \{ \{i, i+1 \} \, | \, i \in \mathbb{Z} \}$. A coloring $\Gamma:V \to C $ is periodic if there exists $\ell \in \mathbb{N}$ such that $\Gamma(i) = \Gamma(i + \ell)$ for all $i \in \mathbb{Z}$. 
Given $n \in \mathbb{N}$ colors, every $\mathbf{m}$-perfect coloring with $n$ colors of an infinite path graph is periodic,~\cite{puzyninaPeriodicityGeneralizedTwodimensional2009}. Interestingly, the period is not directly related to the number of used colors, see Example~\ref{ex:path} for a perfect coloring with two colors which is $3$-periodic. Note that this is not generally valid in the case of more complicated lattices -- see Section~\ref{s:application} for examples of perfect and aperiodic colorings of the square grid.

We summarize the question of existence of a perfect periodic coloring of infinite path graph in the following lemma.

\begin{lem}\label{l:path}
    Let $ n \in \mathbb{N} $ and $\mathbf{m} \in \mathbb{N}_{0}^{n \times n}$ be an irreducible matrix which has row and column sums equal to $2$ and let $m_{ij} =0$ if and only if $m_{ji}=0$ for all $i,j = 1, \ldots, n$. Then there exists a unique (up to its shift and a reflection) $\mathbf{m}$-perfect coloring of a path graph which is periodic. 
\end{lem} 
\begin{proof}
     Our proof relies on a graph-theoretical argument. Since the row and column sums of $\mathbf{m}$ are equal to 2 and since $\mathbf{m}$ is irreducible, we conclude that $\mathbf{m}$ has only $0$ and $1$ entries. Indeed, there can not be an entry greater than $2$ due to the sum condition, and any $2$ entry creates a block in the matrix thus violating the irreducibility condition. It follows that $\mathbf{m}$ is an adjacency matrix of a strongly connected, directed graph $\vec{H}$ with constant in- and out-degree $2$. Due to the assumption of sign symmetry of $\mathbf{m}$ (see Lemma~\ref{l:symmetry}) we can observe the following: if there is an arc in $ \vec{H} $ from $i$ to $j$, then there is an arc from $j$ to $i$. 
     
     It follows immediately from the properties of $ \vec{H} $ that there are two possible configurations of $ \vec{H} $ -- it is either a directed cycle graph with two-way arcs, denoted by $ \vec{H}_1 $, or it is a strongly connected path graph with self-loops at the end vertices, denoted by $\vec{H}_2$. In both of the cases, the graph $ \vec{H} $ possesses two directed Eulerian cycles which differ only by their directions. Given one particular configuration of the graph $ \vec{H} $, the Eulerian path defines a sequence of colors which -- if repeated on the infinite path graph -- creates a coloring of the infinite path graph.  This coloring is $n$-periodic provided $ \vec{H} $ is in the form of $\vec{H}_1$ and it is $2n$-periodic in the case of $\vec{H}_2$.
     
     It remains to show that the coloring is unique with respect to the shift and the reflection of the infinite path graph. Pick an arbitrary color, say $c$, and assign it to a vertex $i$ of the infinite path. There are then two ways how to color the neighbors of $i$ in accordance with the rules of the perfect coloring, they differ by an exchange of the left and the right neighbor. From this point, the matrix $\mathbf{m}$ (alternatively the graph $\vec{H}$) uniquely determines the coloring of the remaining vertices in both directions. Different choice of either the color $c$ or the vertex $i$ results in a shifted coloring. Different choice of the left and right neighbors of $i$ results in a reflected coloring. 
     \end{proof}
\begin{rem}
   The assumption of irreducibility of the matrix $\mathbf{m}$ in Lemma~\ref{l:path} can be relaxed. However, the resulting coloring is then not unique. Indeed, if $\mathbf{m}$ is reducible, the argument in the proof of Lemma~\ref{l:path} can be repeated for each of the irreducible blocks of $\mathbf{m}$. The number of distinct colorings then corresponds to the number of the blocks. 
   
   The sign-symmetry and the condition on the row and column sums are necessary conditions -- and in the context of Lemma~\ref{l:path} also sufficient -- conditions ensuring that $\mathbf{m}$ is a matrix of some perfect coloring. 
\end{rem}
We conclude this paragraph with an illustration of Lemma~\ref{l:path}.
\begin{ex}\label{ex:path-grids}
    Let 
    \begin{equation}\label{e:path}
    \mathbf{m}_1 = \left[ 
        \begin{array}{ccc}
        0 & 1 & 1 \\
        1 & 0 & 1 \\
        1 & 1 & 0
        \end{array} 
        \right],
        \qquad 
    \mathbf{m}_2 = \left[ 
        \begin{array}{ccc}
        1 & 1 & 0 \\
        1 & 0 & 1 \\
        0 & 1 & 1
        \end{array} 
        \right], 
    \end{equation}
    be two matrices which satisfy the assumptions of Lemma~\ref{l:path}. 
    The matrix $\mathbf{m}_1$ is an adjacency matrix of a directed cycle graph depicted in Figure~\ref{f:ex:path:cycle-directed}. An Eulerian path gives us a 3-periodic sequence of colors which creates the $\mathbf{m}_1$-perfect coloring of an infinite path graph, see Figure~\ref{f:ex:path:cycle-graph}. 

    The matrix $\mathbf{m}_2$ is an adjacency matrix of a strongly connected path graph with loops at the end vertices which is depicted in Figure~\ref{f:ex:path:path-directed}. The Eulerian path generates a 6-periodic sequence which is a $\mathbf{m}_2$-periodic coloring of an infinite path graph, see Figure~\ref{f:ex:path:path-graph}. 
    
    \begin{figure}[ht]
         \centering
         \begin{subfigure}[b]{0.2\textwidth}
             \centering
             \includegraphics[width=.5\textwidth]{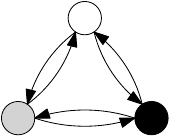}
             \caption{}
             \label{f:ex:path:cycle-directed}
         \end{subfigure}
         \hfill
         \begin{subfigure}[b]{0.75\textwidth}
             \centering
             \includegraphics[width=\textwidth]{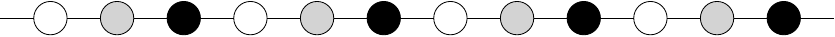}
             \caption{}
             \label{f:ex:path:cycle-graph}
         \end{subfigure}
        \hfill 
        \begin{subfigure}[b]{0.2\textwidth}
             \centering
             \includegraphics[width=\textwidth]{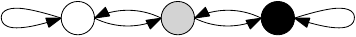}
             \caption{}
             \label{f:ex:path:path-directed}
         \end{subfigure}
         \hfill
         \begin{subfigure}[b]{0.75\textwidth}
             \centering
             \includegraphics[width=\textwidth]{fig/patt-ex-path-path-graph.pdf}
             \caption{}
             \label{f:ex:path:path-graph}
         \end{subfigure}
        \caption{Colorings and graphs constructed in Example~\ref{ex:path}. The panels~\ref{f:ex:path:cycle-directed} and~\ref{f:ex:path:path-directed} depict strongly connected graphs with adjacency matrices $\mathbf{m}_1$ and $\mathbf{m}_2$, respectively, which are defined in~\eqref{e:path}. These matrices represent two qualitatively different graphs admissible by the Eulerian paths of the graph in Figure~\ref{f:ex:path:cycle-directed} and in Figure~\ref{f:ex:path:path-directed} generate perfect -- and periodic at the same time -- coloring from Figure~\ref{f:ex:path:cycle-graph} and Figure~\ref{f:ex:path:path-graph}, respectively. }
            \label{f:ex:path:existence}
    \end{figure}
\end{ex}
\subsection{Regular two-dimensional grids}
As examples of two-dimensional grids we discuss square, hexagonal, and triangular grid. 
\begin{defin}
    Let $T$ be a regular tessellation of an Euclidean plane $\mathbb{R}^2$ with squares/triangles/hexagons. A square/triangular/hexagonal grid is a graph $G=(V,E)$ where $V$ is the set of vertices of the tessellation $T$ and $E$ is the set of edges of the tessellation $T$. 
\end{defin}

We first examine how perfect colorings relate to periodic colorings (recall that in the case of one-dimensional lattice -- the infinite path -- every perfect solution is periodic, \cite{puzyninaPeriodicityGeneralizedTwodimensional2009}). In comparison to the general exposure of perfect stationary solutions on $k$-regular graphs in Section~\ref{s:perfect}, the notion of periodicity is graph-specific. We start with the definition of a periodic coloring of a square grid which can be readily extended to a hexagonal and a triangular grid. 
\begin{defin}\label{d:square}
    Let $\Gamma$ be a coloring of a regular square grid $G$. The coloring $\Gamma$ is \textit{periodic}, if there exist two linearly independent vectors $v_1, v_2 \in \mathbb{Z}^2, v_1, v_2 \neq 0$, such that $\Gamma(i) = \Gamma(t_{v_1}(i))$ and $\Gamma(i) = \Gamma(t_{v_2}(i))$ for all $i \in V$, in which $t_v$ is a translation by a vector $v$. If necessary, we call the particular coloring $(v_1,v_2)$-periodic.
\end{defin}
\begin{rem}\label{r:tri}
   The vertices of the triangular grid which originates from the tiling with equilateral triangles can be shifted so that their coordinates are integers, see Figure~\ref{f:tri}. This allows us to extend the definition of a periodic solution (Definition~\ref{d:square}) to the case of the triangular grid. 
\end{rem}
\begin{figure}
    \centering
    \begin{subfigure}[b]{.45\textwidth}
    \includegraphics[width=\linewidth]{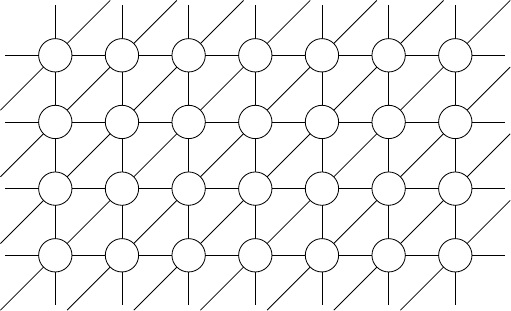}
    \caption{Layout of a triangular grid such that the coordinates of the vertices are integers, see Remark~\ref{r:tri}}
    \label{f:tri}
    \end{subfigure}
   \begin{subfigure}[b]{.45\textwidth}
       \includegraphics[width=\linewidth]{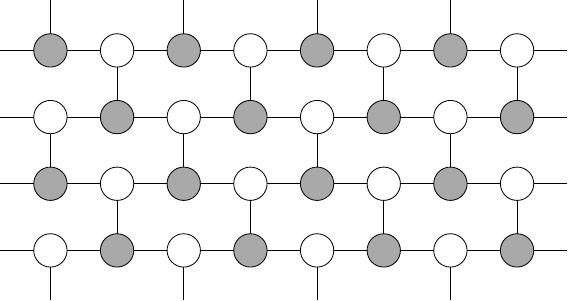}
    \caption{Alternative layout of a hexagonal grid in which the vertices have coordinates in $\mathbb{Z}^2$ mentioned in Remark~\ref{r:hex}. The gray and white colors form two partitions of the grid.} 
    \label{f:hex}
   \end{subfigure} 
   \caption{Alternative layouts of the triangular and the hexagonal grid which serve to introduce periodic colorings via Definition~\ref{d:square}.}
\end{figure}
\begin{rem}\label{r:hex}
There exists a natural placement of vertices of a hexagonal grid to the coordinates $\mathbb{Z}^2$, see Figure~\ref{f:hex}. The periodicity of a coloring on a hexagonal grid can be then defined through a pair of linearly independent nonzero vectors $v_1, v_2 \in \mathbb{Z}^2$ as in Definition~\ref{d:square} with additional condition that 
\begin{equation}\label{e:hex:per}
    (v_{1})_{x}=(v_{1})_{y}=(v_{2})_{x}=(v_{2})_{y} = 0 \quad \mathrm{mod} \quad 2, 
\end{equation}
in which $v_1 = ((v_{1})_{x},(v_{1})_{y}), v_2 = ((v_{2})_{x},(v_{2})_{y})$. In particular, the translation is invariant for the two partitions of the hexagonal grid, see again Figure~\ref{f:hex}. 
\end{rem}

Having the crucial definition of periodic colorings at hand (Definition~\ref{d:square} and Remarks~\ref{r:tri},~\ref{r:hex}) we can now formally show that every periodic coloring on square/triagonal/hexagonal grid can be expressed as an isomorphic coloring. The form of the theorem is not straightforward since not every periodic coloring is exactly an isomorphic coloring. We claim that given a periodic coloring $\widetilde\Gamma$, we can define a perfect coloring $\Gamma$ such that it preserves the distinction of colors of $\widetilde\Gamma$ while possibly adding new ones.     

\begin{thm}\label{t:periodic-is-perfect}
   Let $G_{k}$ be either a square grid ($k=4$), or a triangular grid ($k=6$), or a hexagonal grid ($k=3$) and let $\widetilde\Gamma:V \to \widetilde C$ be a periodic coloring of $G_{k}$. Then there exists perfect $\Gamma: V \to C$, such that $\widetilde\Gamma(i) \neq \widetilde\Gamma(j)$ implies $\Gamma(i) \neq \Gamma(i)$ for each pair of vertices $i,j \in V$.
\end{thm}
\begin{proof}
    The proof relies on a construction of $\Gamma$ given by a periodic coloring $\widetilde\Gamma$. 
    Let $\Gamma$ be a periodic coloring, there exist two linearly independent vectors $v_1, v_2 \in \mathbb{Z}^2$ such that the coloring $\Gamma$ is invariant with respect to the shift in the directions $v_1$ and $v_2$. Note that in the case of hexagonal grid the vectors $v_1$ and $v_2$ satisfy the condition~\eqref{e:hex:per} from Remark~\ref{r:hex}. The vectors $v_1$ and $v_2$ form a parallelogram which can tile the plane. Place the parallelogram such that one of its corners coincides with a grid vertex and without loss of generality pick two adjacent sides of the parallelogram. Now, define a subset of vertices $V$ as the vertices coinciding with the parallelogram with the exception of the two chosen sides; denote it by $V_f$ (in which $f$ stands for finite). Assign each vertex from $V_f$ a unique color. See Figure~\ref{f:t:per:grid} for an example of the the placement of the parallelogram and the choice of the vertices $V_f$. It is now straightforward that $\widetilde\Gamma(i) \neq \widetilde\Gamma(j)$ implies $\Gamma(i) \neq \Gamma(j)$ for each pair of vertices $i,j \in V_{f}$.

    It remains to show that the extension of this finite coloring is a perfect coloring. First, extend the coloring to each tile of the parallelogram tiling; note that the choice of vertices in the parallelogram ensures that each of the vertices of $G_{k}$ can be uniquely assigned a color. Denote this coloring by $\Gamma$. It follows from the construction that $\Gamma$ is also a periodic coloring and vertex of each color has uniquely defined structure of colors in its neighborhood. See Figure~\ref{f:t:per:finite} for an illustration. 
\end{proof}
\begin{figure}
    \centering
    \begin{subfigure}[b]{.65\textwidth}
    \includegraphics[width=\textwidth]{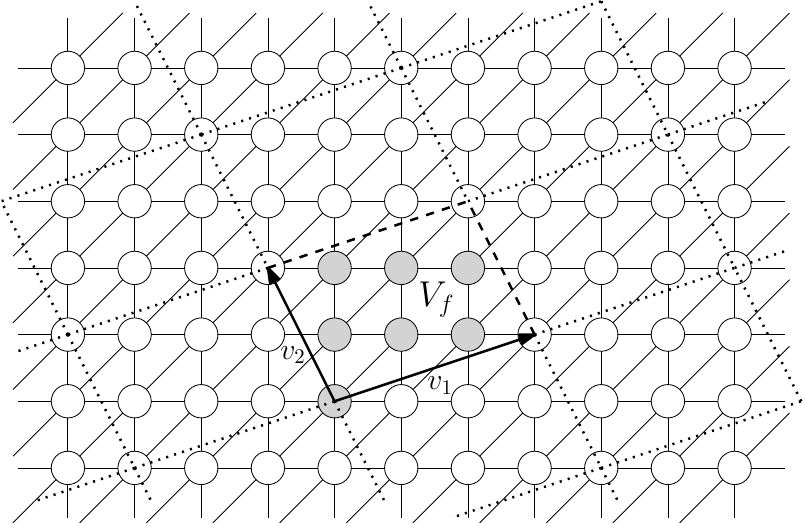}
    \caption{Illustrative example complementing the proof of Theorem~\ref{t:periodic-is-perfect}. We depict a triangular grid with layout as in Remark~\ref{r:tri} in which we assume a periodic coloring given by two linearly independent vectors $v_1 = (3,1)^\top, v_2  = (-1,2)^\top$ (Definition~\ref{d:square}). The two vectors form a parallelogram. The vertices in the parallelogram with the exception of the ones lying on the dashed sides form a finite structure by which we can tile the grid. The vertices -- in the proof of Theorem~\ref{t:periodic-is-perfect} denoted by $V_f$ -- are highlighted by a gray color.}
    \label{f:t:per:grid}    
    \end{subfigure}
    \hspace{2pt}
    \begin{subfigure}[b]{.3\textwidth}
   \includegraphics[width=\textwidth]{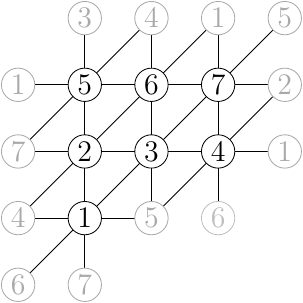}
    \caption{Illustrative example complementing the proof of Theorem~\ref{t:periodic-is-perfect}. The graph depicts the vertices $V_f$ together with their assigned coloring with seven colors. The gray vertices indicate the immediate neighborhood of the vertices $V_f$ together with the coloring given by the extension of the finite coloring.}
    \label{f:t:per:finite} 
    \end{subfigure}
    \caption{Graphs illustrating the proof of Theorem~\ref{t:periodic-is-perfect}.}
\end{figure}

The proof of Theorem~\ref{t:periodic-is-perfect} immediately yields that all periodic perfect colorings with a prescribed period are mergers of the constructed perfect coloring $ \Gamma $.

\begin{cor} \label{c:per-to-perfect}
Let $G_{k}$ be a square, or a triangular, or a hexagonal grid. Let $ v_{1}, v_{2} \in \mathbb{Z}^{2} $ be given, linearly independent and let $ \Gamma $ be the $ \mathbf{m} $-perfect coloring of $ G_{k} $ from the proof of Theorem~\ref{t:periodic-is-perfect}, then all periodic colorings of $ G_{k} $ with period given by $ v_{1}$, $ v_{2} $ which are perfect are mergers of $ \Gamma $.
\end{cor}

We illustrate Corollary~\ref{c:per-to-perfect} in the following example.

\begin{ex}\label{ex:periodic-is-perfect}
   Let us have a four-color perfect coloring $\Gamma$ of a square grid with the matrix
   \begin{equation}\label{e:ex:perfect-is-periodic-matrix}
   \mathbf{m} =
   \begin{bmatrix}
       0 & 2 & 2 & 0 \\
       2 & 0 & 0 & 2 \\
       2 & 0 & 0 & 2 \\ 
       0 & 2 & 2 & 0
   \end{bmatrix}, 
   \end{equation}
   and with the periodic layout from Figure~\ref{f:ex:isomorphic-is-periodic-isom}, i.e., the linearly independent vectors defining the periodicity are $v_1=(2,0)$ and $v_2=(0,2)$. 

    There is exactly one $4$-color $(v_1,v_2)$-periodic coloring which is the coloring $\Gamma$. Other $4$-color $(v_1,v_2)$-periodic colorings can be obtained by grid automorphisms or color permutations. 
    There are two unique $(v_1,v_2)$-periodic colorings denoted by $\widetilde\Gamma_1$, $\widetilde\Gamma_2$ with three colors; other $(v_1,v_2)$-periodic colorings can be obtained through grid automorphisms or color permutations.  
    There are three unique $(v_1,v_2)$-periodic colorings denoted by $\widetilde\Gamma_3$, $\widetilde\Gamma_4$, and $\widetilde\Gamma_5$ with two colors; again other two-color $(v_1,v_2)$-periodic colorings can be obtained through grid automorphisms or color permutations. 
    There is a unique monochromatic coloring with respect to color permutations.
    
  The colorings $\widetilde\Gamma_s$, $s = 1, \ldots,5$, possess the property that $\widetilde\Gamma_s(i) \neq \widetilde\Gamma_s(j)$ implies $\Gamma(i) \neq \Gamma(j)$. Let us define each of $\widetilde\Gamma_s$ by a function $\phi_s : C \to \tilde{C}$ ($ C = \{ 1, 2, 3, 4 \} $ and $ \tilde{C} = \{ \mathrm{white}, \mathrm{black} \} $, or $ \tilde{C} = \{ \mathrm{white}, \mathrm{gray}, \mathrm{black} \} $) such that $\widetilde\Gamma_s = \phi_s \, \circ \Gamma$, specifically:
   \begin{align}\label{e:ex:perfect-is-periodic}
        \begin{array}{rllll}
       \phi_1: & 1 \mapsto \text{black}, & 2 \mapsto \text{white}, & 3 \mapsto \text{gray}, & 4 \mapsto \text{black},\\  
       \phi_2: & 1 \mapsto \text{black}, & 2 \mapsto \text{white}, & 3 \mapsto \text{gray}, & 4 \mapsto \text{white},\\  
       \phi_3: & 1 \mapsto \text{black}, & 2 \mapsto \text{black}, & 3 \mapsto \text{black}, & 4 \mapsto \text{white},\\  
       \phi_4: & 1 \mapsto \text{black}, & 2 \mapsto \text{black}, & 3 \mapsto \text{white}, & 4 \mapsto \text{white},\\  
       \phi_5: & 1 \mapsto \text{black}, & 2 \mapsto \text{white}, & 3 \mapsto \text{white}, & 4 \mapsto \text{black},
       \end{array}
   \end{align}
see Figure~\ref{f:ex:perfect-is-periodic-per}. There are several interesting remarks and connections:
\begin{enumerate}[label=\itshape{(\roman*)}]
\item For the given perfect periodic coloring $ \Gamma $, we are able to find multiple periodic colorings $ \widetilde{\Gamma}_{s} $, $ s = 1, \ldots, 5 $, with less colors.
\item The colorings $\widetilde{\Gamma}_1, \widetilde{\Gamma}_4, \widetilde{\Gamma}_5$ are perfect themselves. Their matrices $\mathbf{m}_1, \mathbf{m}_4, \mathbf{m}_5$ can be obtained via~\eqref{e:perfect-merger} using $\phi_1, \phi_4$, and $\phi_5$, respectively. 
\item Some periodic colorings (namely, $\widetilde{\Gamma}_2, \widetilde{\Gamma}_3$) need not be perfect themselves. However, Theorem~\ref{t:periodic-is-perfect} presents a way how to represent them as perfect colorings with more colors.
\end{enumerate}
   \begin{figure}
       \centering
       \begin{subfigure}[b]{.3\textwidth}
       \includegraphics[width=\linewidth]{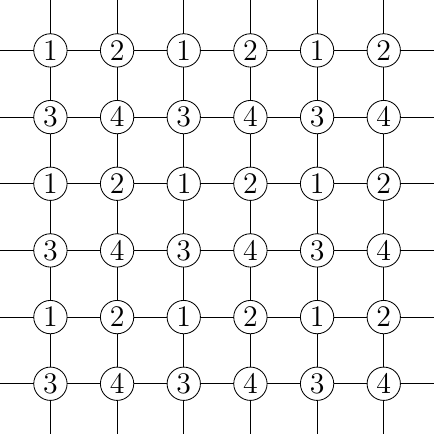}
       \caption{Perfect layout $\Gamma$ of periodic coloring.}
       \label{f:ex:isomorphic-is-periodic-isom}
       \end{subfigure}       
       \hspace{5pt}
       \begin{subfigure}[b]{.3\textwidth}
          \includegraphics[width = \textwidth]{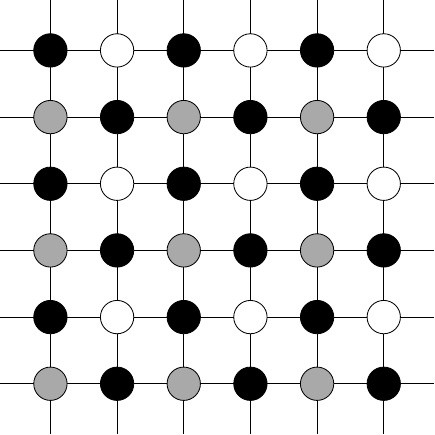}
          \caption{Periodic coloring $\widetilde\Gamma_1$.}
          \label{f:patt-ex-periodic-is-isomorphic-per-1}
       \end{subfigure}       
       \hspace{5pt}
       \begin{subfigure}[b]{.3\textwidth}
          \includegraphics[width = \textwidth]{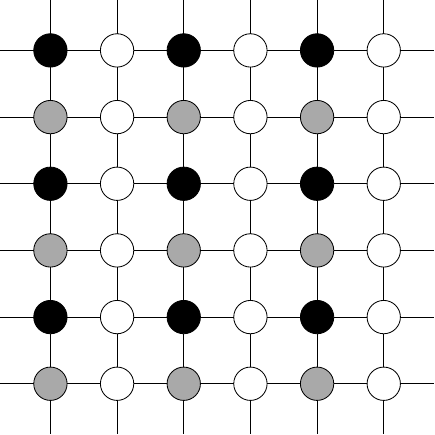}
          \caption{Periodic coloring $\widetilde\Gamma_2$.}
          \label{f:patt-ex-periodic-is-isomorphic-per-2}
       \end{subfigure}
        \hspace{5pt}

       \begin{subfigure}[b]{.3\textwidth}
          \includegraphics[width = \textwidth]{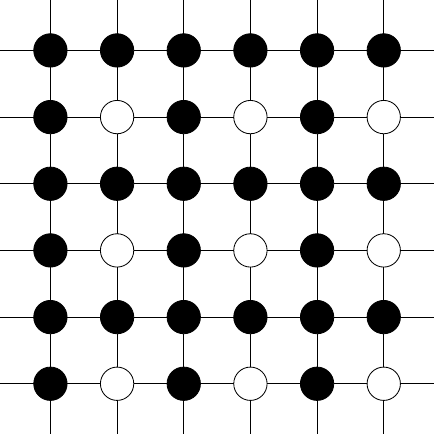}
          \caption{Periodic coloring $\widetilde\Gamma_3$.}
          \label{f:patt-ex-periodic-is-isomorphic-per-3}
       \end{subfigure}
        \hspace{5pt}
       \begin{subfigure}[b]{.3\textwidth}
          \includegraphics[width = \textwidth]{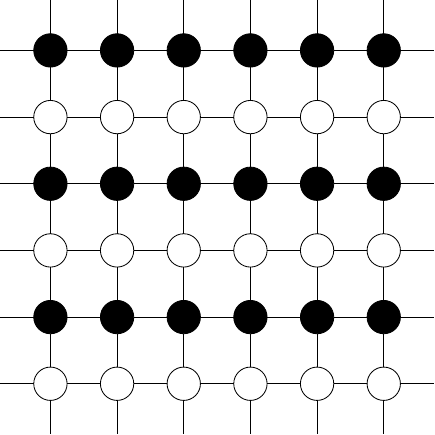}
          \caption{Periodic coloring $\widetilde\Gamma_4$.}
          \label{f:patt-ex-periodic-is-isomorphic-per-4}
       \end{subfigure}
       \hspace{5pt}
       \begin{subfigure}[b]{.3\textwidth}
          \includegraphics[width = \textwidth]{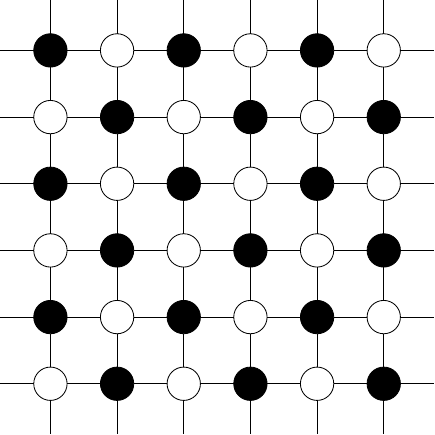}
          \caption{Periodic coloring $\widetilde\Gamma_5$.}
          \label{f:patt-ex-periodic-is-isomorphic-per-5}
       \end{subfigure}
       \caption{The perfect coloring $\Gamma$ and periodic colorings from Example~\ref{ex:periodic-is-perfect} obtained via the $\widetilde\Gamma_s = \phi_s \, \circ \, \Gamma$ in which $\phi_s$ is defined in~\eqref{e:ex:perfect-is-periodic} and the perfect coloring $\Gamma$ has the matrix~\eqref{e:ex:perfect-is-periodic-matrix} with the layout as in Figure~\ref{f:ex:isomorphic-is-periodic-isom}.}
       \label{f:ex:perfect-is-periodic-per}
   \end{figure}
\end{ex}

Focusing on the third remark in the preceding Example~\ref{ex:periodic-is-perfect}, we revealed that periodic coloring of a grid need not be perfect itself. However, considering periodic \emph{stationary solutions} (when periodicity of a solution on the grid is defined analogically as the periodicity of a coloring, see Definition~\ref{d:square}) we see that each periodic stationary solution of~\eqref{e:GDE} on square/triangular/hexagonal grid is necessarily perfect as well.

\begin{thm}\label{t:periodic-is-perfect-sol}
    Let $G_{k}$ be a square, or a triangular, or a hexagonal grid, and let~\eqref{e:GDE} be defined on $G_{k}$. Then for each linearly independent $v_1, v_2 \in \mathbb{Z}^2$ (satisfying~\eqref{e:hex:per} provided $G_{k}$ is a hexagonal grid) there exists a matrix $\mathbf{m} \in \mathbb{N}_0^{n \times n}$ such that every $(v_1,v_2)$-periodic stationary solution of~\eqref{e:GDE} (generalize Definition~\ref{d:square}) is $\mathbf{m}$-perfect. 
\end{thm}

\begin{proof}
The construction in the proof of Theorem~\ref{t:periodic-is-perfect} assigns each admissible pair of linearly independent vectors $v_1, v_2 \in \mathbb{Z}^2$ a perfect coloring $\Gamma$ with colors $C$ and the matrix $\mathbf{m}$. There additionally exists a surjection $v : C \to \mathrm{Im}(u)$ such that 
\[
v_{\Gamma(i)} = u_i \quad \text{for every} \quad i \in V.
\]
This follows from the consequence of Theorem~\ref{t:periodic-is-perfect}. Thus, $u$ is $\mathbf{m}$-perfect by Definition~\ref{d:perfect-solution}.
\end{proof}

\begin{rem}\label{r:periodic-is-perfect-sol}
    The choice of $\mathbf{m}$ in the proof of Theorems~\ref{t:periodic-is-perfect} and~\ref{t:periodic-is-perfect-sol} is not necessarily optimal. Let the assumptions of Theorem~\ref{t:periodic-is-perfect-sol} hold and let $u$ be a $(v_1,v_2)$-periodic solution. If there exists a set of colors $\widetilde{C}$ and a one-to-one map $v: \widetilde{C} \to \mathrm{Im}(u)$ such that the coloring $\widetilde{\Gamma}: V \to \widetilde{C}$ defined by $\widetilde{\Gamma}(i) = v^{-1}(u_i)$ is $\widetilde{\mathbf{m}}$-perfect, then naturally $u$ is also $\widetilde{\mathbf{m}}$-perfect. It however follows from the construction in the proof of Theorem~\ref{t:periodic-is-perfect} that $\widetilde{\Gamma}$ is a merger of $\Gamma$ and Lemmas~\ref{l:perfect-merger} and~\ref{l:ss-merger} apply. 
    
    Subsequently, if one wants to study $(v_1,v_2)$-periodic stationary solutions, it is beneficial to examine the finite systems~\eqref{e:finite-stationary} with $\widetilde{\mathbf{m}}$ being the matrices of the mergers $\widetilde{\Gamma}$ of $\Gamma$ -- $\Gamma$ originating from the vectors $v_1, v_2$. See Section~\ref{s:application}.
\end{rem}
Motivated by the preceding statements and examples, we discuss in the rest of this section the following issues:
\textit{
\begin{enumerate}[label=\arabic*.]
\item We study a possible number of non-equivalent perfect colorings of a given regular two-dimensional grid. In particular, we show that there can be an uncountable number of distinct $\mathbf{m}$-perfect colorings of the grid by two colors.
\item Theorem~\ref{t:periodic-is-perfect}, Corollary~\ref{c:per-to-perfect}, and Example~\ref{ex:periodic-is-perfect} show that every periodic coloring of a given regular two-dimensional grid is perfect itself or it can be represented by a perfect periodic coloring with more colors. We also discuss the contrary, if there exists (and under what conditions) a perfect coloring which is not periodic.
\end{enumerate}
}

To start with the number of perfect colorings, we show that even for two colors (i.e., for $ \mathbf{m} \in \mathbb{N}_{0}^{2 \times 2}$) there can generally be uncountable many non-equivalent perfect colorings of an underlying two-dimensional grid. Let us firstly precisely specify what colorings are equivalent.

\begin{defin}\label{d:equivalent-colorings}
    Let $G_{k}$ be a square, or a triangular, or a hexagonal grid and let $\Gamma_1$ and $\Gamma_2$ be its two colorings. The colorings $\Gamma_1$, $\Gamma_2$ are called \textit{equivalent} if there exists an automorphism of the grid $\varphi: V \to V$ such that 
    \[
    \Gamma_1(i) = \Gamma_2(\varphi(i))
    \]
    holds for all $i \in V$. The colorings $\Gamma_1$, $\Gamma_2$ are \textit{non-equivalent} if they are not equivalent. 
\end{defin}

\subsubsection{Square grid}

The square grid is a 4-regular graph, $k=4$. Two-color perfect colorings, $n=2$, have been thoroughly examined in~\cite{axenovichMultipleCoveringsInfinite2003}. We summarize the results in the next theorem in which we describe the $\mathbf{m}$-perfect colorings by a two-dimensional vector $(m_{11},m_{22})$ according to the following key
\[
\mathbf{m} = 
\begin{bmatrix}  
m_{11} & m_{12} \\
m_{21} & m_{22}
\end{bmatrix} 
= 
\begin{bmatrix}
m_{11} & 4-m_{11} \\
4-m_{22} & m_{22}
\end{bmatrix},
\]
where $ m_{11} $ represents the number of alike neighbors of a vertex with the first color and $ m_{22} $ represents the number of alike neighbors of a vertex with the second color.

\begin{thm}[\cite{axenovichMultipleCoveringsInfinite2003}]\label{t:square:grid}
    Let $G_{4}$ be a square grid and assume $n=2$. There generally exists an uncountable number of non-equivalent two-color perfect colorings on the square grid. The specific number of non-equivalent colorings depends on a particular form of the matrix $\mathbf{m}$ in the following way:
\begin{center}
\begin{NiceTabular}{|w{c}{1.1cm}|c|c|c|c|}\hline
\rule[0mm]{5mm}{0cm}
\diagbox{$m_{11}$}{$m_{22}$}&
\Block{}{$0$}&\Block{}{$1$}&\Block{}{$2$}&\Block{}{$3$}\\
\hline
$0$   & $\exists!$  & $\nexists$  & $\exists!$  & $\exists!$  \\
\hline
$1$   & ``        & $\exists!$  & $\exists!$  & $\infty$  \\
\hline
$2$   & ``        & ``        & $\infty$    & $2 \times \exists$  \\
\hline
$3$   & ``        & ``        & ``        & $\exists!$           \\
\hline

\hline
\end{NiceTabular}
\end{center}
in which the symbol $\nexists$ means that the corresponding coloring does not exist, $\exists!$ indicates the existence of precisely one coloring with respect to the automorphisms of the grid, $ 2 \times \exists$ means that there exist two non-equivalent colorings, and $\infty$ stands for infinite number of pairwise non-equivalent colorings.

In particular:
\begin{enumerate}[label=\arabic*.]
    \item The table is symmetric, i.e., the colors can be exchanged.
    \item The infinite number of colorings actually indicates uncountable number of colorings.  
    \item If $m_{11}=4$ or $m_{22}=4$, then there exists only a monochromatic coloring with the respective color.
\end{enumerate}
\end{thm}

To discuss the second question on the existence of an aperiodic perfect coloring of the square grid, assume a general number of colors $n \ge 2$. We have a general necessary and sufficient condition on the matrix $\mathbf{m}$ to guarantee the existence of the aperiodic $\mathbf{m}$-perfect coloring.

\begin{thm}(\cite[Theorem 1]{puzyninaPeriodicityGeneralizedTwodimensional2009}) \label{t:rectangular-det}
Let $ n \in \mathbb{N} $ and $\mathbf{m} \in \mathbb{N}_{0}^{n \times n}$ be such that $\sum_{j=1}^n m_{ij}=4$ holds for all $i = 1,2, \ldots, n$. There exists an aperiodic $\mathbf{m}$-perfect coloring $\Gamma$ of a rectangular grid if and only if $\det(\mathbf{m}) = 0$. 
\end{thm}

The proof of Theorem~\ref{t:rectangular-det} in~\cite{puzyninaPeriodicityGeneralizedTwodimensional2009} relies on the method of $R$-prolongable words. Similar result for two colors can be found in~\cite{heikkilaPerfectCoveringsTwoDimensional2022} which utilizes algebraic argument. Unfortunately, the condition on the determinant does not provide a direct information about the number of perfect colorings in general. Nevertheless, at least for the case of two colors $(n=2)$ the non-periodicity seems to correspond to the existence of uncountable number of colorings, see Theorem~\ref{t:square:grid} and the entries corresponding to the singular matrices 
\begin{equation}\label{e:square-grid-matrices}
\mathbf{m}_1 = 
\begin{bmatrix}
1 & 3 \\
1 & 3
\end{bmatrix} , \quad  
\mathbf{m}_2 = 
\begin{bmatrix}  
2 & 2 \\
2 & 2
\end{bmatrix} , \quad 
\mathbf{m}_3 = 
\begin{bmatrix}
3 & 1 \\
3 & 1
\end{bmatrix} .
\end{equation}

\subsubsection{Triangular grid}
Triangular grid is a $6$-regular graph. The two-color perfect colorings were investigated in~\cite{vasilevaDistanceRegularColorings2014} and we summarize the results in the following theorem. Similarly to the case of the square grid we use a key 
\[
\mathbf{m} = \left[ 
\begin{array}{cc}  
m_{11} & m_{12} \\
m_{21} & m_{22}
\end{array} 
\right] = \left[ 
\begin{array}{cc}  
m_{11} & 6-m_{11} \\
6-m_{22} & m_{22}
\end{array} 
\right].
\]

\begin{thm}(\cite{vasilevaDistanceRegularColorings2014, vasilevaCompletelyRegularCodes2022})\label{t:triangular:grid}
Let $G_6$ be a triangular grid and assume $n=2$. There generally exists an uncountable number of non-equivalent two-color perfect colorings on the triangular grid. The specific number of non-equivalent colorings depends on a particular form of the matrix $\mathbf{m}$ in the following way:

\begin{center}
\begin{NiceTabular}{|w{c}{1.1cm}|c|c|c|c|c|c|}\hline
\rule[0mm]{5mm}{0cm}
\diagbox{$m_{11}$}{$m_{22}$}&
\Block{}{$0$}&\Block{}{$1$}&\Block{}{$2$}&\Block{}{$3$}&\Block{}{$4$}&\Block{}{$5$}\\
\hline
$0$&$\nexists$ &$\nexists $&$\nexists$&$\exists!$&$\infty$&$\exists!$\\
\hline
$1$&``&$\nexists$&$\nexists$& $\infty$ &$\exists!$&$\nexists$\\
\hline
$2$&``&``&$\infty$&$2 \times \exists$ & $2 \times \exists$ &$\nexists$\\
\hline
$3$& `` &``&``&$\exists!$&$\nexists$&$\nexists$\\
\hline
$4$& ``&``&``&`` &$\exists!$&$\nexists$\\
\hline
$5$& ``&``&``&``&`` &$\nexists$\\
\hline
\end{NiceTabular}
\end{center}
(see Theorem~\ref{t:square:grid} for the explanation of the symbols). 

In particular:
\begin{enumerate}
    \item The table is symmetric, i.e., the colors can be exchanged.
    \item The infinite number of colorings actually indicates uncountable number of colorings. 
    \item If $m_{11}=6$ or $m_{22}=6$, then there exists only a monochromatic coloring with the respective color.
\end{enumerate}
\end{thm}

There is a typo in the text of~\cite{vasilevaDistanceRegularColorings2014} claiming that in the case $(m_{11}, m_{22}) = (1,3)$ there is exactly one coloring. The existence of an infinite number of colorings in this case is later clarified in~\cite{vasilevaCompletelyRegularCodes2022}. We remark that the results for $m_{11}+m_{22} \leq 2$, and $m_{11}+m_{22} \geq 8$, $m_{11},m_{22} \neq 6$ -- with the exception of $(m_{11},m_{22}) \neq (4,4)$ -- can also be found in~\cite{taranenkoMetricPropertyPerfect2021}.

In the virtue of Theorem~\ref{t:rectangular-det} there is a claim connecting the aperiodicity of the coloring and the spectral properties of the matrix $\mathbf{m}$. 

\begin{thm}(\cite[Theorem 2]{puzyninaPeriodicityGeneralizedTwodimensional2009}) \label{t:triangular-det}
Let $ n \in \mathbb{N} $ and $\mathbf{m} \in \mathbb{N}_{0}^{n \times n}$ be such that $\sum_{j=1}^n m_{ij}=6$ holds for all $i = 1,2, \ldots, n$. There exists an aperiodic $\mathbf{m}$-perfect coloring $\Gamma$ of a triangular grid if and only if $\det(\mathbf{m}+2I)= 0$, in which $I$ is an identity matrix. 
\end{thm}

Although the proof of Theorem~\ref{t:triangular-det} does not provide a precise information about the number of existing colorings, there is an observation similar to the one in the case of the square grid. Namely, the matrices in the form
\[
\mathbf{m}_s = 
\begin{bmatrix}
    s & 6 - s \\
    2 + s & 4 - s 
\end{bmatrix}, \quad s = 0, 1, \ldots, 4,
\]
satisfy
\[
\det(\mathbf{m}_s + 2 I )= \det\left( 
\begin{bmatrix}
    2 + s & 6 - s \\
    2 + s & 6 - s
\end{bmatrix} \right) = 0, \quad s = 0, 1, \ldots, 4, 
\]
i.e., the condition from Theorem~\ref{t:triangular-det} (cf. the table in Theorem~\ref{t:triangular:grid}).  

\subsubsection{Hexagonal grid}
The hexagonal grid possesses slightly different qualities compared to other infinite graphs considered in this paper. Although it is a vertex transitive graph, it is not shift transitive since there are two types of qualitatively different vertices, see Remark~\ref{r:hex} and Figure~\ref{f:hex}. Since the hexagonal grid is a $3$-regular graph, the corresponding matrix $\mathbf{m}$ has necessarily the following form:
\[
\mathbf{m} = 
\begin{bmatrix}
m_{11} & m_{12} \\
m_{21} & m_{22}
\end{bmatrix} 
= 
\begin{bmatrix}
m_{11} & 3-m_{11} \\
3-m_{22} & m_{22}
\end{bmatrix}.
\]
The two color patterns were investigated in~\cite{avgustinovichCompletelyRegularCodes2016} in the context of perfect regular codes. We summarize the number of non-equivalent $\mathbf{m}$-perfect colorings of the hexagonal grid in Theorem~\ref{t:hexagonal:grid}.  
\begin{thm}(\cite{avgustinovichCompletelyRegularCodes2016})\label{t:hexagonal:grid}
Let $G_3$ be a hexagonal grid and assume $n=2$. There generally exists an uncountable number of non-equivalent two-color perfect colorings on the hexagonal grid. The specific number of non-equivalent colorings depends on a particular form of the matrix $\mathbf{m}$ in the following way:

\begin{center}
\begin{NiceTabular}{|w{c}{1.1cm}|c|c|c|}\hline
\rule[0mm]{5mm}{0cm}
\diagbox{$m_{11}$}{$m_{22}$}&
\Block{}{$0$}&\Block{}{$1$}&\Block{}{$2$}\\
\hline
$0$& $\exists!$&$\nexists$& $\infty$ \\
\hline
$1$&``&$\infty$ & $\exists!$\\
\hline
$2$& `` & `` &$\infty$\\
\hline
\end{NiceTabular}
\end{center}
(see Theorem~\ref{t:square:grid} for the explanation of the symbols). 

In particular:
\begin{enumerate}
    \item The table is symmetric, i.e., the colors can be exchanged.
    \item The infinite number of colorings actually indicates uncountable number of colorings. 
    \item If $m_{11}=3$ or $m_{22}=3$, then there exists only a monochromatic coloring with the respective color.
\end{enumerate}
\end{thm}

Similarly to the case of the rectangular grid, there is a general statement about the sufficient condition of the periodicity of a perfect coloring. 
\begin{thm}(\cite[Theorem 3]{puzyninaPeriodicityGeneralizedTwodimensional2009}, \cite[Proposition 5]{puzyninaPeriodicityPerfectColorings2011}) \label{t:hexagonal-det}
Let $ n \in \mathbb{N} $ and $\mathbf{m} \in \mathbb{N}_{0}^{n \times n}$ be such that $\sum_{j=1}^n m_{ij}=3$ holds for all $i = 1,2, \ldots, n$. There exists an aperiodic $\mathbf{m}$-perfect coloring $\Gamma$ of a hexagonal grid if and only if $\det(\mathbf{m}^2-I)= 0$, in which $I$ is an identity matrix.
\end{thm}

The proof of Theorem~\ref{t:hexagonal-det} 
does not however provide a direct information about the cardinality of the sets of $\mathbf{m}$-perfect colorings. Nevertheless, the condition $\det (\mathbf{m}^2 - I)=0$ is satisfied for
\[
\mathbf{m}_1 = \left[ 
\begin{array}{cc}  
0 & 3 \\
1 & 2
\end{array} 
\right], \quad  
\mathbf{m}_2 = \left[ 
\begin{array}{cc}  
1 & 2 \\
2 & 1
\end{array} 
\right], \quad 
\mathbf{m}_3 = \left[ 
\begin{array}{cc}  
2 & 1 \\
1 & 2
\end{array} 
\right] , \quad 
\mathbf{m}_4 = \left[ 
\begin{array}{cc}  
2 & 1 \\
3 & 0
\end{array} 
\right],
\]
which are all cases in which there is an uncountable number of $\mathbf{m}$-perfect colorings according to the results~\cite{avgustinovichCompletelyRegularCodes2016}.

\section{Application}\label{s:application}
We conclude this paper with a final demonstration of perfect stationary solutions of~\eqref{e:GDE}. As a starting point let us consider periodic solutions on a square grid corresponding to the colorings from Example~\ref{ex:periodic-is-perfect}. According to Theorem~\ref{t:periodic-is-perfect}, each of the colorings $\widetilde\Gamma_i$ is periodic with the vectors $\mathbf{v}_1 = (2,0)$ and $\mathbf{v}_2 = (0,2)$ (see Figure~\ref{f:ex:perfect-is-periodic-per}).

Let us consider a lattice differential equation~\eqref{e:GDE} on the square lattice/grid with bistable cubic nonlinear term
\[
f(s;a) = s(1-s)(s-a), \qquad a \in (0,1).
\]

While investigating the periodic solutions on a square grid with $\mathbf{v}_1$ and $\mathbf{v}_2$ of~\eqref{e:GDE}, it is actually sufficient to focus on the system
\begin{equation}\label{e:applications-stationary-problem}
0 = d(\mathbf{m} \, \cdot \, v - 4 v) + F(v),
\end{equation}
see~\eqref{e:finite-stationary} in Theorem~\ref{t:main} and Theorem~\ref{t:periodic-is-perfect-sol}, with the matrix 
\[
  \mathbf{m} =
   \begin{bmatrix}
       0 & 2 & 2 & 0 \\
       2 & 0 & 0 & 2 \\
       2 & 0 & 0 & 2 \\ 
       0 & 2 & 2 & 0
   \end{bmatrix}, 
\]
i.e., the coloring in Figure~\ref{f:ex:isomorphic-is-periodic-isom} in Example~\ref{ex:periodic-is-perfect}. At the same time the colorings $\widetilde\Gamma_1, \widetilde\Gamma_4, \widetilde\Gamma_5$ are perfect. It is suitable to particularly mention another perfect coloring, the homogeneous coloring which we denote by $\widetilde{\Gamma}_6$. The colorings have matrices 
\[
\mathbf{m}_1 = 
\begin{bmatrix}
    0 & 2 & 2 \\
    4 & 0 & 0 \\
    4 & 0 & 0
\end{bmatrix},
\quad
\mathbf{m}_4 =
\begin{bmatrix}
    2 & 2 \\
    2 & 2
\end{bmatrix},
\quad
\mathbf{m}_5 = 
\begin{bmatrix}
    0 & 4 \\
    4 & 0
\end{bmatrix}, 
\quad
\mathbf{m}_6 = 
\begin{bmatrix}
    4 
\end{bmatrix}.
\]
see Lemma~\ref{l:perfect-merger}.
Although the colorings $\widetilde\Gamma_2$, $\widetilde\Gamma_3$ use finite number of colors and are periodic, they are not perfect. The crucial observation here is that if we pose a restriction to the equation~\eqref{e:applications-stationary-problem} such that
\begin{itemize}
    \item $v_1 = v_4$, then~\eqref{e:applications-stationary-problem} becomes $0 = d(\mathbf{m}_1 \, \cdot \, v - 4 v) + F(v)$ where $v \in \mathbb{R}^3$,
    \item $v_1 = v_2$, $v_3 = v_4$, then~\eqref{e:applications-stationary-problem} becomes $0 = d(\mathbf{m}_4 \, \cdot \, v - 4 v) + F(v)$ where $v \in \mathbb{R}^2$,
    \item $v_1 = v_4$, $v_2 = v_3$, then~\eqref{e:applications-stationary-problem} becomes $0 = d(\mathbf{m}_5 \, \cdot \, v - 4 v) + F(v)$, where $v \in \mathbb{R}^2$,
    \item $v_1 = v_2 = v_3 = v_4$, then~\eqref{e:applications-stationary-problem} becomes $0 = d(\mathbf{m}_6 \, \cdot \, v - 4 v) + F(v) = F(v)$, where $v \in \mathbb{R}$,
\end{itemize}
which follow immediately from Lemma~\ref{l:ss-merger}.
In other words, if there is a solution $ v $ of the corresponding
system with the respective matrices $\mathbf{m}_1, \mathbf{m}_4, \mathbf{m}_5$, $ \mathbf{m}_6 $ then the vector $(v_1, v_2, v_3, v_1)$, $(v_1,v_1,v_2,v_2)$, $(v_1, v_2, v_2, v_1)$, or $(v_1, v_1, v_1, v_1)$ is a solution to~\eqref{e:applications-stationary-problem}, respectively. 

For example, setting $a=0.4$ and $d=0.005$, there are $3^4 = 81$ solutions of the problem~\eqref{e:applications-stationary-problem}. Let us take a look at the subsets of solutions. First, there are $3$ homogeneous stationary solutions which are represented by a monochromatic coloring. Next, the coloring $\widetilde\Gamma_5$ represents $3^2 = 9$ solutions out of which $3$ are homogeneous. There are thus in total $9-3 = 6$ solutions of ~\eqref{e:applications-stationary-problem} which can be represented by $\widetilde\Gamma_5$ and by no merger of it.
There are two possible layouts of the coloring $\widetilde\Gamma_4$ which are just rotated by the angle $\tfrac{\pi}{2}$. Each layout represents $2^3=9$ solutions. Nevertheless, the three homogeneous stationary solutions are also represented by $\widetilde\Gamma_4$ and its rotation. This leads to the total of $2 \cdot (9-3) = 12$ solutions represented by $\widetilde\Gamma_1$ and no merger of it. Finally, there are two possible layouts of the coloring $\widetilde\Gamma_4$ which are just rotated by the angle $\tfrac{\pi}{2}$. Each layout represents $3^3 = 27$ solutions. The homogeneous coloring and the coloring $\widetilde\Gamma_5$ are mergers of $\widetilde\Gamma_1$. We have thus $2 \, \cdot \, (27-9) = 36$ colorings represented by $\widetilde\Gamma_1$ and no merger of it. 

There is however one exciting observation. The matrices $\mathbf{m}, \mathbf{m}_1$ and $\mathbf{m}_4$ are singular with zero determinant. According to Theorem~\ref{t:rectangular-det}, there exist $
\mathbf{m}$-, $\mathbf{m}_1$-, and $\mathbf{m}_4$-perfect colorings of the square grid which are not periodic, see Figure~\ref{f:applications} for various examples of non-periodic $\mathbf{m}_1$- and $\mathbf{m}_4$-perfect colorings.
   \begin{figure}
       \centering
       \begin{subfigure}[b]{.48\textwidth}
          \includegraphics[width = \textwidth]{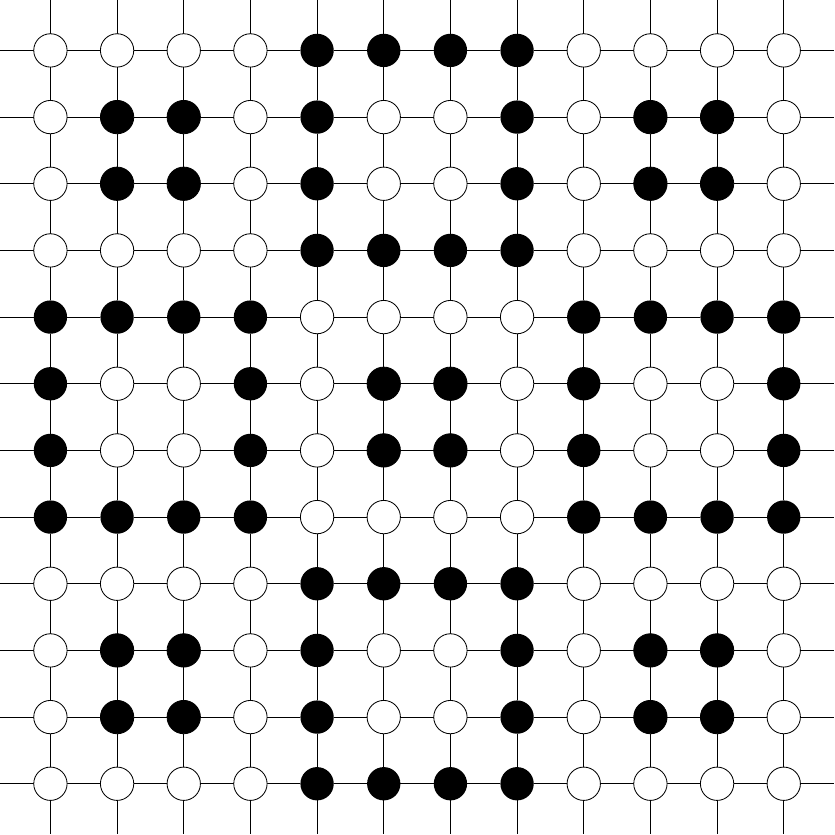}
          \caption{}
          \label{f:patt-square-grid-22-ver-1}
       \end{subfigure}       
       \hspace{5pt}
       \begin{subfigure}[b]{.48\textwidth}
          \includegraphics[width = \textwidth]{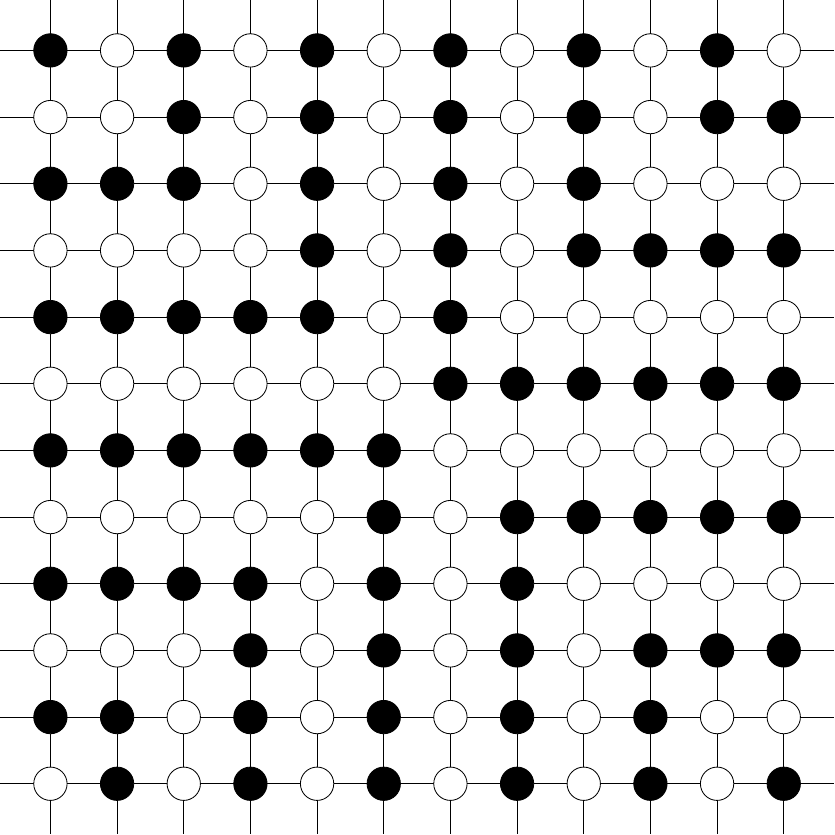}
          \caption{}
          \label{f:patt-square-grid-22-ver-2}
       \end{subfigure}

       \begin{subfigure}[b]{.48\textwidth}
          \includegraphics[width = \textwidth]{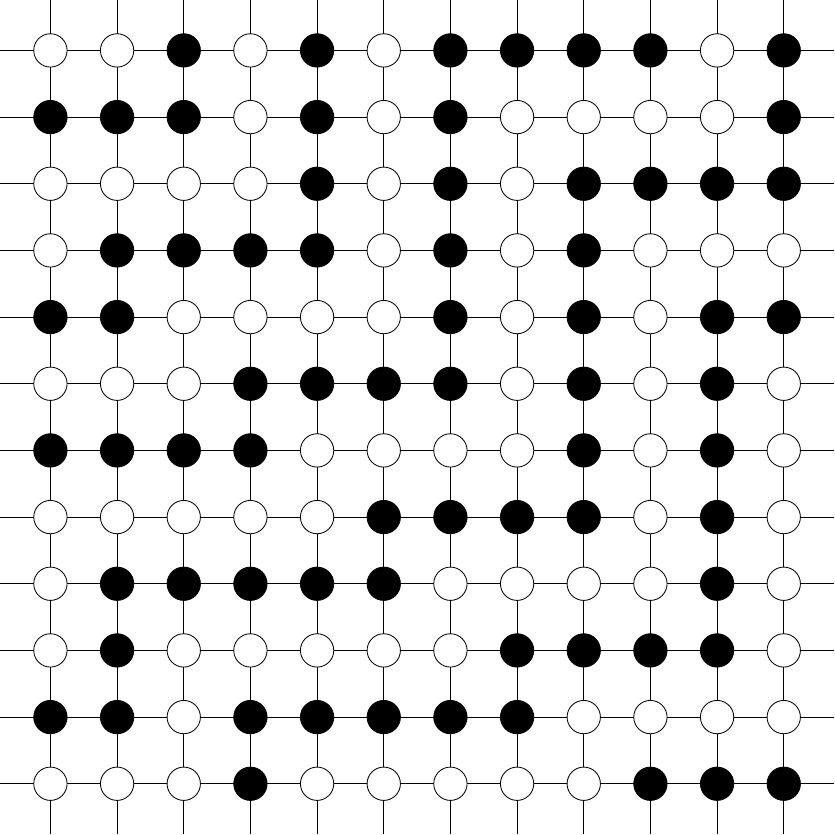}
          \caption{}
          \label{f:patt-square-grid-22-ver-4}
       \end{subfigure}
       \hspace{5pt}
       \begin{subfigure}[b]{.48\textwidth}
          \includegraphics[width = \textwidth]{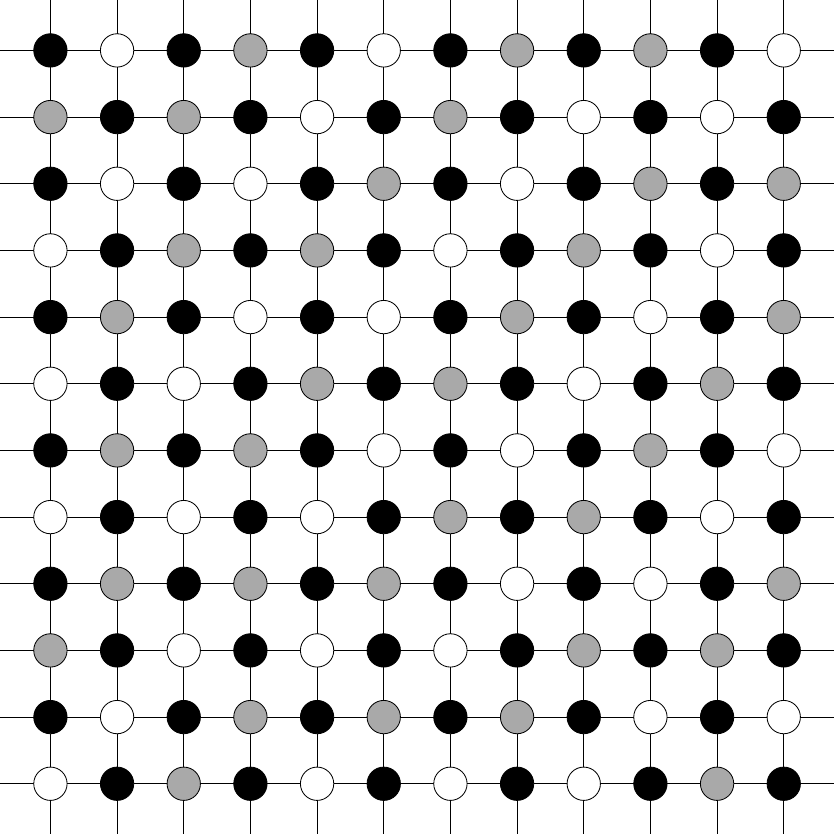}
          \caption{}
          \label{f:patt-square-grid-3col-ver-1}
       \end{subfigure}
       \caption{Examples of possible layouts for stationary solutions described by matrices $\mathbf{m}_1, \mathbf{m}_4$ in Section~\ref{s:application}. The layout~\ref{f:patt-square-grid-22-ver-1} is periodic but the period can be arbitrarily inflated by making the tiles bigger. The layout~\ref{f:patt-square-grid-22-ver-2} is not periodic in the sense of Definition~\ref{d:square} although there is some regularity present. The layout~\ref{f:patt-square-grid-22-ver-4} is also aperiodic and underlines the strategy of constructing uncountably many perfect colorings -- hence, stationary solutions -- for $\mathbf{m}_1$ by encoding doubly-infinite sequence of zeros and ones into right- and up turns of the monochromatic paths, see Theorem~\ref{t:square:grid}. The layout in~\ref{f:patt-square-grid-3col-ver-1} uses similar idea as~\ref{f:patt-square-grid-22-ver-4} but with three colors. }
       \label{f:applications}
   \end{figure}

\section{Conclusion}
We explored a specific class of stationary solutions of the reaction-diffusion equations on regular graphs~\eqref{e:GDE}, the perfect stationary solutions. We aimed to introduce the notion, examine its elementary properties, and connect them to existing concepts as periodic stationary solutions on regular grids. 

There are various directions in which the topic of perfect stationary solutions and the finite range solutions of reaction-diffusion equations can be extended. We include some open questions which originate from the direction of dynamical systems and algebraic equations.
\begin{itemize}
    \item How to characterize and how prevalent are the finite range stationary solutions of reaction-diffusion equations which are not perfect, cf Example~\ref{ex:tree}?
    \item Assume that the function $F$ in~\eqref{e:GDE} is element-wise polynomial. Can the idea of reducing the stationary problems via mergers of perfect colorings, Lemma~\ref{l:ss-merger} be used to decompose the stationary problem into smaller ones utilizing algebraic techniques such as Gr\"{o}bner bases? 
    \item Rigorously examine an equivalent statement to Theorem~\ref{t:periodic-is-perfect-sol} in higher-dimensional regular grids. This should mainly consist of proper definition or regular grids in higher dimensions and generalizing Theorem~\ref{t:periodic-is-perfect}.
\end{itemize}
Naturally, there is plethora of open problems regarding perfect colorings of graphs.
\begin{itemize}
    \item How to characterize aperiodic perfect colorings with regular structure such as the one in Figure~\ref{f:patt-square-grid-22-ver-2}?
    \item Given a matrix $\mathbf{m}$ and a graph $G$, how to fully characterize all $\mathbf{m}$-perfect colorings of the graph $G$?
    \item Existence and the number of colorings of higher-dimensional grids. 
\end{itemize}
\section{Acknowledgement}
The authors acknowledge the support of the project GA22-18261S by the Czech Science Foundation.
Both authors are thankful to Petr Stehl\'{i}k for numerous discussions about the topic and his insightful remarks to the manuscript.

\end{document}